\documentclass[12pt]{report}
\usepackage[margin=1in]{geometry}% Change the margins here if you wish.

\usepackage{amsmath}

\usepackage{amsmath,amsthm,amssymb}

\usepackage{graphicx}

\usepackage{color}
\usepackage{enumerate}
\usepackage{multicol}

\usepackage{hyperref}

\usepackage[nameinlink]{cleveref}
\usepackage[all]{xy}
\usepackage{ marvosym }
\usepackage{comment}

\usepackage{halloweenmath}

\newtheorem{thm}{Theorem}[section]
\newtheorem*{thm*}{Theorem}

\newtheorem{prop}[thm]{Proposition}
\newtheorem*{prop*}{Proposition}

\newtheorem{lemma}[thm]{Lemma}
\newtheorem*{lemma*}{Lemma}

\newtheorem{assm}[thm]{Assumption}
\newtheorem*{assm*}{Assumption}

\newtheorem{corollary}[thm]{Corollary}
\newtheorem*{corollary*}{Corollary}

\newtheorem{claim}[thm]{Claim}
\newtheorem*{claim*}{Claim}

\newtheorem{defi}[thm]{Definition}
\newtheorem*{defi*}{Definition}

\newtheorem{fact}[thm]{Fact}
\newtheorem*{fact*}{Fact}

\newtheorem*{example*}{Example}

\newtheorem*{examples*}{Examples}

\newtheorem*{remark*}{Remark}

\newtheorem{conjecture}[thm]{Conjecture}
\newtheorem*{conjecture*}{Conjecture}

\newtheorem*{goal*}{Goal}

\newtheorem*{subgoal*}{Subgoal}

\newtheorem*{question*}{Question}

\newtheorem*{problem*}{Problem}

\DeclareMathOperator{\acl}{acl}
\DeclareMathOperator{\dcl}{dcl}

\DeclareMathOperator{\RM}{RM}

\DeclareMathOperator{\DM}{DM}

\newcommand{\s}{\subseteq}

\begin{document}

\title{Strongly minimal reducts of ACVF} % You should make this the title of your project.
\author{Santiago Pinzon} % Replace this with your name.
 \date{October 31, 2022} 
 %I have this commented out so that the program will use whatever today's date is.  You can specify a particular date (or use this for other info if you need) in the {} or leave the {} empty for no date (or other extra info) to appear in the title

\maketitle % If you don't want a formal title, you can erase this and the bits above.  You will need to find a different way to include the same information.

% Some projects will just be a list of problems, and in that case you may wish to number them:

\begin{center}
    \large\textbf{Abstract}
\end{center}

Let $\mathbb K=(K,+,\cdot,v,\Gamma)$ be a valued algebraically closed field of characteristic and $(G,\oplus)$ be a $\mathcal K$-interpretable group that is either locally isomorphic to $(K,+)$ or to $(K,\cdot)$. Then if $\mathcal G=(G,\oplus,\ldots)$ is a strongly minimal non locally modular structure intepretable in $\mathbb K$, it interprets a field.

We also present an strategy for proving the same without the assumption of having a definable group operation.

This document is the PhD thesis of the author and it was advised by professors Alf Onshuus and Assaf Hasson.
\tableofcontents

\chapter{Introduction and Preliminaries}\label{introductionAndPreliminaries}

\section{Introduction}

Given any field $k$ there are many families of curves in $k^2$ that one can define. For example one may define the family of all lines on $k^2$. This, together with the points of $k^2$, form a plane geometry in the sense of Chapter 2.1 of \cite{algGeoArt}. In here Artin proved a converse for this. Namely if one starts with a set $E$ and an incidence system of lines and points on $E$ satisfying some geometrical axioms, then one can to build a field $k$ such that the lines of $E$ are given by linear equations on $k$. 

One can wonder whether this result in some other settings where one has a good notion of dimension and a big enough family of curve -in the sense of that dimension-. In this way, in \cite{rabinovich}, Rabinovich proved that if $k$ is an algebraically closed valued field and $\mathcal D=(D,\ldots)$ is some structure whose universe is $\mathbb A ^1 (k)$ and whose definable sets are constructible sets, then if $\mathcal D$ has a big enough definable family of curves contained in $D^2$, then $\mathcal D$ interprets an infinite field. 

An example of such a family is the family of all lines contained in $\mathbb A^1 \times \mathbb A^1$. This is a two dimensional family of curves, and in this case is not hard to see that one can recover an infinite field. For example if one consideres the sub-family of lines passing trough $(0,0)$ this can be identified with $\{l_m:m\in k\}\cup l_\infty$. Where for $m\in K$ $l_m$ is the line of equation $y=mx$ and $l_\infty$ is the line of equation $x=0$. In this case one has that $l_m\circ l_n=l_{m\cdot  n}$ so one can recover the multiplicative group of $k$, $\mathbb G_m$ from composition between elements of that subfamily. With some more work we can also recover the additive group $\mathbb G_a$ and also the action of $\mathbb G_m$ on $\mathbb G_a$.

%In modern model theoretical language the analogous of this is Zilber's trichotomy. It was stated by Zilber in the 1970s and says that any strongly minimal set is either trivial, locally modular or it interprets a field  (precise definitions are given in Section \ref{preliminariesModelTheory}).

In late 1970s, Zilber abstracted this principle in the following conjecture:

\begin{conjecture}(Zilber's Trichotomy Principle)

If $\mathcal D$ is a strongly minimal structure and there is a big enough $\mathcal D$-definable family of plane curves, then $\mathcal D$ interprets an infinite field.
\end{conjecture}

In the conjecture, ``big enough'' will be read for us as ``morley rank $2$''. The family of lines in the plane is an example.

Ravinovich' result is a very particular case of this conjecture where $\mathcal D$ is some reduct of the full field structure on the affine line of an algebraically close field.

This conjecture was proved false by Hrushovski in \cite{hruNew}. He built a class of examples of strongly minimal sets that are not trivial or locally modular but does not interpret any field.  However, the principle of Zilber's trichotomy has still an important role in modern model theory.

In \cite{hrz} Hrushovski and Zilber defined (1-dimensional) Zariski geometries and proved that Zilber's trichotomy holds for them. This covers a vast class of examples, generalizing the algebraic case: if $k$ is an algebraically closed field and $N$ are the $k$-points of an algebraically curve (over $k$) then if $\mathcal N$  is the structure with universe $N$ and whose definable subsets of $N^k$ are all the $k$-constructible sets, then $\mathcal N$ is a Zariski Geometry and then it interprets a field.

%In \cite{chatzidakis1999model} Chatzidakis and Hrushovski used Zariski Geometries for proving that the conjecture is true for strongly minimal sets definable in a deferentially closed valued field. That was used for Hrushovki in his prove of Mordell-Lang conjecture in positive characteristic.

There are other settings where the conjecture has been proved true without the use of Zariski Geometries but using intersection theory. That is the case of \cite{HS}. Here Hasson and Sustretov proved that if $D$ has as universe an algebraic curve over an algebraically closed field $k$ and all the definable sets on $D^k$ are definables in the field structure, then Zilber's trichotomy it is true for $D$. Note that it is a generalization of Ravinovich result.

In some other settings the conjecture has proved to be true by using intersection theory coming from continuous open functions. This is the case of \cite{HE}, here is proved that if $(D,+)$ is a definable group in $k$, an o-minimal expansion of a real closed field, and $\mathcal D=(D,+,\ldots)$ is a o-minimal structure that is a reduct of the structure induced by $k$, containing a big enough family of curves, then $\mathcal D$ interprets a field.

There is also cases where the conjecture was proved to be true using intersection theor coming from analytico functions. This is the case of \cite{KR}. Here it is proved that if $D$ has as universe a valued field $K$ with $\text{char} K=0$ and we assume that addition is definable and that all the definable sets of $D^k$ are definable in the valued field structure, then $D$ satisfies Zilber's trichotomy.

This is the setting in which we are interested. In the introduction of \cite{HS} is suggested that their methods should be suitable to be used for proving generalizations of \cite{KR}, for example proving the result for positive characteristic or get rid of the assumption that $+$ is definable on $D$.

In this thesis we work on both possible generalizations:

First we prove that results on \cite{KR} are true even in positive characteristic. This is Theorem \ref{propAdditive}.

Moreover we also prove that if $D$ is an expansion of the multiplicative group $(K\setminus 0,\cdot)$ then it interprets a field if it is not locally modular. 

We expect that our results can be used for proving Zilber's trichotomy for $D$. Assuming that $D$ is any definable subset of an algebraically close valued field $K$ whose Zariski closure is $1$-dimensional and whose definables sets are definable in the valued fields structure. This can be done finding a group interpretable in $\mathcal D$ as in Section \ref{findingAGroup} and then proving that such a group is locally isomorphic either to $(K,+)$ or to $(K\setminus 0,\cdot)$, and then we can use our result to conclude. 

Now we present the structure of the document:

In Chapter \ref{introductionAndPreliminaries} we present the basic preliminaries on model theory and valued fields that we will need. 

In Chapter \ref{someFacts} we prove some facts about definable groups contained in $K$. In Section \ref{goodFamiliesOfCurves} we introduce the notion of good families of curves for a definable group $G$ and prove Theorem \ref{thmFieldGivenFamily} that we will use for defining a field.

In Chapter \ref{additiveCase} we construct a good family of curves if $G$ is the additive group of $K$ so using results on Chapter \ref{someFacts} we prove Theorem \ref{propAdditive}.

In Chapter \ref{multiplicativeCase} we deal with the case in wich $D$ is an expansion of the multplicative group. In here we firs build an interpretable group that is locally isomorphic to $(K,+)$. Then we use results of Chapters \ref{someFacts} and \ref{additiveCase} for interpret a field. 

\section{Preliminaries on Model Theory}\label{preliminariesModelTheory}

Let $\mathcal L$ be a first order language and let $\mathcal P=(P,\ldots)$ be an infinite $\mathcal L$ structure. 

We adopt the usal definitios of $\mathcal P$-definable and $\mathcal P$-intepretable sets. See for example Section 1.3 of \cite{marker}.  Assume moreover that $\mathcal P$ is $\kappa$-saturated, for some big enough cardinal $\kappa$. In particular $\kappa>\omega$.

\begin{defi}
Given $D$ be an interpretable set on $\mathcal P$.

We say that $D$ is strongly minimal if $D$ is infinite and the only definable subsets of $D$ are the finite and the cofinite sets. 

We say that $\mathcal P$ is strongly minimal if $P$ is strongly minimal as a definable set on $\mathcal P$.

If $T$ is a theory we say that $T$ is strongly minimal if $\mathcal P$ is strongly minimal for all $\mathcal P\models T$
\end{defi}
We will use the following notion of dimension:

\begin{defi}
If $X\subseteq P^n$ is a definable set we say that $\RM_{\mathcal P}(X)\geq 0$ if $X$ is non empty and for an ordinal $\alpha$, $\RM_{\mathcal P}(X)\geq \alpha$ if there are $X_1,X_2\ldots$ infinitely many definable subsets of $X$ such that:
\begin{itemize}
    \item $\bigcup X_i=X$, 
    \item for all $i\neq j$, $X_i\cap X_j=\emptyset$  and
    \item for all $i$ and for all $\beta <\alpha$,  $\RM_{\mathcal P}(X_i)\geq \beta$..
\end{itemize}

We say that $\RM_{\mathcal P}(X)=\alpha$ if $\RM_{\mathcal P}(X)\geq \alpha$ and is not the case that $\RM_{\mathcal P}(X)\geq \alpha +1$. 

If there is some $\alpha<\kappa$ such that $\RM(X)=\alpha$, we say that $X$ has bounded morley rank and that $\alpha$ is the morley rank of $X$.

If $\RM(X)=n\in \omega$ then we say that $X$ has finite morley rank.

\end{defi}

The following is well known, a proof can be find in Lemma 6.2.7 of \cite{marker}.

\begin{fact}
If $X,Y$ are subsets of $P^n$, then 
\end{fact}

We omit the subscript if $\mathcal P$ is clear from the context.

\begin{defi}
If $p(x)\in S_n(A)$ we define 
$$\RM(p)=\min\{\RM(X):X \in p\}.$$

For $a\in D^n$ let  $\RM(a/p)=\RM(tp(a/A))$.
\end{defi}

The following is well known (and easy to prove) and can be found in Lemma 6.2.7 of \cite{marker}.

\begin{fact}\label{factGoodDimension}
If $\mathcal P=(P,\ldots)$ is any structure and $X,Y$ are  definable subsets of $P^n$, then:

\begin{enumerate}
    \item If $X\subseteq Y$ then $\RM(X)\leq\RM(Y)$.
    \item $\RM(X\cup Y)=\max(\RM(X),\RM(Y))$.
    \item If $X\neq \emptyset$ then $\RM(X)=0$ if and only if $X$ is finite.
\end{enumerate}
\end{fact}

\begin{defi}
If $\RM(X)=\alpha$ then there is no a partition of $X$ in infinitelly many definable subsets of morley rank greater than $\alpha$. Therefore by compactness there is a maximal natural number $n$ such that there are $X_1,\ldots,X_n$ definable and disjoints subsets of $X$ covering $X$ with $\RM(X_i)=\alpha$.
We define the morley degree of $X$ as $\DM(X)=n$. 

We say that $X$ is sationary if $\DM(X)=1$.
\end{defi}

\begin{lemma}
A set $D$ is strongly minimal if and only if $\RM(D)=\DM(D)=1.$
\end{lemma}

\begin{proof}
Suppose $D$ is strongly minimal, then as $D$ is infinite and $D=\cup_{d\in D}\{d\}$ one has that $\RM X\geq 1$. As each infinite definable subset of $D$ is cofinite it is not the case that there are two infinite and disjoints definable subsets of $D$. It shows that $\RM(X)=1$ and also that $\DM(X)=1$.

Now assume that $\RM X=\DM X=1$, and suppose by contradiction that there is an infinite definable set $Y\subseteq X$ such that $X\setminus Y$ is also infinite. Then as $X=Y\cup (X\setminus Y)$ one has that $\DM(X)\geq 2$, a contradiction.

\end{proof}

The following can be found in page 196 of \cite{marker2017strongly}.
\begin{fact}\label{factDimN}
If $\mathcal P=(P,\ldots)$ is strongly minimal then $\RM(P^n)=n$.
\end{fact}

\begin{corollary}
If $\mathcal P=(P,\ldots)$  is strongly minimal and $X\subseteq P^n$ is $\mathcal{P}$-definable, then $X$ has finite Morley Rank.
\end{corollary}

\begin{proof}
By Fact \ref{factDimN} one has that $\RM P^n=n$ and by clause 1 of Fact \ref{factGoodDimension} one conclude that $\RM X\leq n$, in particular $X$ has finite Morley Rank.
\end{proof}

\begin{defi}
Let $\mathcal P=(P,\ldots)$ be a strongly minimal structure and $X$ be a $\mathcal P$-interpretable set with parameters $A$. We say that a tuple $z=(z_1,\ldots,z_n)\in X^n$ is generic independent (over $A$) if $\RM(z/A)=n\RM(X)$ 
\end{defi}

We now define curves and families of curves. 

\begin{defi}
If $\mathcal P=(P,\ldots)$ is a strongly minimal structure, a plane curve $C$ (of $\mathcal P$) is a $\mathcal P$-definable and one dimensional subset of $P^2$. 

A definable family of plane curves is a $\mathcal M$-definable set $X\subseteq P^{2+n}$ such that for all $a$ in some $\mathcal P$-definable set $Q\subseteq P^n$ one has that $$X_a:=\{(x,y)\in P^2:(x,y,a)\in X\}$$ is a plane curve.

We usually write such a family as $(X_a)_{a\in Q}$.

\end{defi}

\begin{defi}
We say that a family of curves $(X_a)_{a\in Q}$ is almost disjoint if for any $b\in Q$ the set $\{a\in Q: |X_a\Delta X_b|<\infty \}$ is finite.
\end{defi}

From now when we say ``definable family of curves'' we mean ``definable and almost disjoint family of curves''.

Now we define the notion of trivial and locally modular strongly minimal structures.

\begin{defi}
If $\mathcal P$ is strongly minimal we say that $\mathcal P$ is trivial if for any definable set $A\subseteq P^n$ one has that 
\[
\acl(A)=\bigcup_{a\in A} \acl(a).
\]
\end{defi}

\begin{defi}
If $\mathcal P$ is strongly minimal we say that $\mathcal P$ is locally modular if for any $X,Y$ definable subsets of $P^n$, if $X=\acl(X)$ and $Y=\acl(Y)$ then one has that
\[
\RM(X\cup Y)=\RM X + \RM Y - \RM(X\cap Y).
\] 
\end{defi}

Zilbers Trichotomy states:

Let $\mathcal P$ be strongly minimal, then and assume that is no trivial and is no locally modular, then there is an infinite field $k$ interpretable in $\mathcal P$.

In all of its generality it was proved false by Hrushovski on \cite{hruNew}. But we can re state it relative to restricted setting:

Let $T$ be some complete theory 

\begin{conjecture}\label{zilbTric}(Zilber's trichotomy principle relative to $T$)

Let $\mathcal N=(N,\ldots)$ be any model of $T$ and let $P$ be some $\mathcal N$-interpretable set. Let $\mathcal P=(P,\ldots)$ be some $\mathcal L'$-structure over $P$ (for some language $\mathcal L'$) such that any $\mathcal P$-definable set is also $\mathcal N$-definable. Then if $\mathcal P$ is non locally modular, there is an infinite field interpretable in $\mathcal P$.
 
\end{conjecture}
 
There are several instances of this conjecture that has been proved true, more relevant for us are:

\begin{fact} (Theorem 4.3.3 of \cite{HS})
If $T$ is the theory of algebraically closed fields of a fixed characteristic (as defined in Section \ref{preliminariesOnValuedFields}) then Conjecture \ref{zilbTric} is true if we assume that $P$ is an algebraic curve. 
\end{fact}

Recently Castle proved in \cite{CAS} the following:

\begin{fact}
If $\mathcal N$ is an algebraically closed field of characteristic zero, then, Conjecture \ref{zilbTric} is true.
\end{fact}

The following is also true:

\begin{fact} (Theorem  3.17  on \cite{KR})
If $T$ is the theory of algebraically closed valued field of characteristic zero (as defined in Section \ref{preliminariesOnValuedFields}) then Conjecture \ref{zilbTric} is true if we assume that $\mathcal N=(N,+,\cdot,v,\Gamma)$, $ P=N$ and addition is definable in $\mathcal P$.
\end{fact}

We will use techniques of both, \cite{HS} and \cite{KR} for almost all of our work.

\section{Group Configurations}

In this section we introduce the main tools in order to interpret groups and fields in strongly minimal structures. 

From now we fix $\mathcal N=(N,\ldots)$ an strongly minimal structure.

\begin{defi}
A $d$-dimensional group configuration for $\mathcal N$ over a set of parameters $A$ is a $6$-tuple  $\mathfrak g=(a_1,a_2,a_3,b_1,b_2,b_3)$ where each $b_i$ and $a_i$ are tuples of elements of $N$ such that:

\begin{itemize}
    \item $\RM \mathfrak g= 3d+3$
    \item $\RM(\alpha,\beta/A)=\RM(\alpha/A)+\RM(\beta/A)$ for all $\alpha\neq \beta\in \mathfrak g$,
    \item  $\RM(b_i/A)=d$ for  $i=1,2,3$, 
    \item  $\RM(a_i/A)=1$ for $i=1,2,3$, 
    \item  $\RM(b_1,b_2,b_3/A)=3d$, 
    \item $\RM(b_1, a_2, a_3)=\RM(b_2,a_1,a_3)=\RM(b_3,a_1,a_2)=d+1$.
\end{itemize}
\end{defi}

\begin{defi}
If $G$ is an interpretable group of dimension $d$, a group configuration of $G$ is 
$$\mathfrak g_G=(a,b\cdot a, c\cdot b\cdot a ,b,c,cb)$$ for some choice of $a,b,c$ generic independent elements of $G$
\end{defi}

\begin{defi}
We say that $\mathfrak g=(a_1,a_2,a_3,b_1,b_2,b_3)$ is a reduced group configuration for $\mathcal N$ if it is a group configuration and in addition if $a'_i\in \acl(a_i)$ are such $\mathfrak g'=(a'_1,a'_2,a'_3,b_1,b_2,b_3)$ is still a group configuration, then $a_i\in \acl(a'_i)$  
\end{defi}

If $\mathfrak g_1$ is another group configuration we say that $\mathfrak g$ and $\mathfrak g_1$ are interalgebraic if the correspondent coordinates satisfies $\acl(a)=\acl(a')$.
The following is due to Hrushovski (\cite{hrushovski1986contributions}). The precise statement we need can be find in Facts 4.4 and  4.6 of \cite{HS}.

\begin{fact}\label{groupConfig}

Let $\mathcal N$ be an strongly minimal structure, then if $\mathfrak g$ is a group configuration for $\mathcal N$ (over some set of parameters), there is a minimal group $G$, an strongly minimal set $X$ and a faithful action of $G$ on $X$ all of the data interpretable in $N$.

In addition $\mathfrak g$ is reduced, then a generic group configuration of $G$ is interalgebraic with $\mathfrak g$. In particular $\RM G=d$.

\end{fact}

The following is also due to Hrushovski \cite{hrushovski1986contributions}.

\begin{fact}\label{fieldConfig}
Let $G$ be a group interpretable in $\mathcal N$ acting transitively and faithfully on a strongly minimal set $X$. Assume $\RM(G)=2$ then there is a field structure definable in $X$ and $G$ is isomorphic to the semidirect product of the multiplicative and the additive group of such a field.\\

In particular if $\mathfrak g$ is a group configuration of dimension $2$ then $\mathcal N$ interprets a field.
\end{fact}

So it makes sense to define:

\begin{defi}
A field configuration is a $2$-dimensional group configuration.
\end{defi}

\section{Preliminaries on valued fields}\label{preliminariesOnValuedFields}

We start with some basic definitions.

\begin{defi}
Given a field $\mathbb K=(K,+,\cdot,0,1)$ we treat it as a first order structure in the language of rings $\mathcal L_R=\{+,\cdot,0,1\}$. Let $ACF_p$ be the first order theory of algebraically closed fields of characteristic $p$. That is $ACF_p$ is the theory of fields of characteristic $p$:

That is $(K,+,0)$ is an abelian group:
\begin{itemize}
    \item $\forall x (x+0=x)$,
    \item $\forall x \exists y (x+y=0)$,
    \item $\forall x,y,z((x+y)+z=x+(y+z))$ and
    \item $\forall x,y (x+y=y+x)$.
\end{itemize}
The product ($\cdot$) is a binary operation defined on $K\times K$ such that $(K\setminus 0,\cdot,1)$ is also an abelian group:

\begin{itemize}
    \item $\forall x (x\cdot 1=x)$,
    \item $\forall x (x\neq 0\implies \exists y (x\cdot y=1))$,
    \item $\forall x,y,z[(z\neq 0\wedge y\neq 0 \wedge z\neq 0)\implies (x\cdot y)+z=x+(y+z)]$ and
    \item $\forall x,y (x\cdot y=y\cdot x)$.
\end{itemize}

And product distributes over addition:

\begin{equation*}
    \forall x,y,z(x\cdot(y+z)=x\cdot y + x\cdot z,
\end{equation*}

together with the scheme of axioms given by $$\forall a_0\ldots\forall a_n \exists x (a_0+a_1 x +\ldots a_n x^n=0)$$  for all $n\geq 1$. 

And we say that $\text{char} (K)=p$, so if $p=0$ we add the scheme of axioms $$\phi_n:=1+1+1\cdots+1\neq 0$$ where $1$ is added $n$ times.

If $\text{char} (K) =p>0$ we add the axioms $\phi_n$ for $n<p$ plus the axiom $\neg \phi_p$.
\end{defi}
\begin{defi}
Given a field $(K,+,\cdot,0,1)$ we say that a set $X\subseteq K^n$ is zariski closed if there is a set of polynomials $\mathfrak a \subseteq K[x_1,\ldots,x_n]$ such that for all $x\in K^n$, $x\in X$ if and only if $f(x)=0$ for all $f\in \mathfrak a$. 
\end{defi}

\begin{fact} (It follows for example from Theorem 3.2.2 of \cite{marker})
The theory $ACF_{p}$ is complete and strongly minimal
\end{fact}

%\begin{fact}
%If $\mathbb K=(K,+,\cdot,0,1)$ is a model of $ACF_p$, then $\RM_{\mathbb M}(X)=\dim (\cl(X))$ where $\cl(X)$ is the Zariski closure of $X$. 
%\end{fact}

We will need:

\begin{fact}\label{bezoutTheorem}(Bezout Theorem)

Let $k$ be algebraically closed field and let $F(x,y)$ and $G(x,y)$ be polynomials with coefficients on $k$ with no common non constant divisors, then if $V$ is the set of zeros for $F$ and $W$ is the set of zeros for $G$ then $V\cap W$ is finite and it has less than $\deg(F)\deg(G)$ points counting multiplicities. If we considerate the closures of $V$ and $W$ in the projective space $\mathbb P^2$ then the number of points intersection is exactly $\deg(F)\deg(G)$ (counting multiplicities).
\end{fact}

\begin{defi} 
If $K$ is a field, a valuation on $K$ is an ordered abelian group $\Gamma$ togheter with a valuation map $v:K\to \Gamma\cup\{\infty\}$ (where $\infty$ is an extra element such that $\infty>\gamma$ and $\infty+\gamma=\gamma+\infty=\infty$ for each $\gamma\in \Gamma$) such that for all $x,y\in K$, $v$ satisfies:

\begin{enumerate}
    \item $v(x)=\infty$ if and only if $x=0$.
    \item $v(x\cdot y)=v(x)+v(y)$.
    \item $v(x+y)\geq \min\{v(x),v(y)\}$. 
\end{enumerate}
\end{defi}
If $K$ is valued, the valuation ring is $\mathcal O_K=\{x\in K:v(x)\geq 0\}$. It is easy to see that is a subring of $K$ and its only maximal ideal is: $\mathfrak m=\{x\in K:v(x)>0\}$. The residue field is $k=\mathcal O_K/\mathfrak m$.

We treat a valued field as a two sorted structure $(K,\Gamma,v)$ 
where $K$ is a field, $\Gamma$ is an algebraically closed ordered group and $v:K\to \Gamma\cup\{\infty\}$ is the valuation. Let $\mathcal L_{R,v}$ be the correspondent two sorted language.
\begin{defi}
Let ACVF be the first order theory in the language $\mathcal L_{R,v}$ saying that $K$ is an algebraically closed field and $v$ is a valuation into an ordered abelian group $\Gamma$:

In ACVF we have the axioms that says that $\Gamma$ is an abelan group plus the extra axioms

$(\Gamma,\leq)$ is a linear order:

\begin{itemize}
\item $\forall x,y\in \Gamma(x\leq y \vee y\leq x)$,
\item $\forall x,y,z\in \Gamma(x\leq y \wedge y\leq x\implies x\leq z)$, 
\item $\forall x,y\in \Gamma(x\leq y\wedge y\leq x\implies x=y)$ and 
\item $\forall x\in \Gamma (x\leq x)$.

\end{itemize}

And addition on $\Gamma$ is compatible with $\leq$:

$$\forall x,y,z\in \Gamma (x\leq y\implies x+z\leq y+z)$$

\end{defi}

From now we fix $\mathbb K=(K,+,\cdot,\Gamma,v)$ a model of ACVF.

%In this document we  prove (some particular cases of) the following theorem:

%\begin{thm}\label{thmGrupo}

%If $(G,\cdot)$ is a $\mathbb{K}$-definable group and $X\subseteq G^2$ is a $\mathbb K$-definable set that is not a boolean combination of cosets of subgroups of $G$, then the structure $(G,\cdot,X)$  interprets an infinite field.
%\end{thm}

\begin{defi}
An open ball is a subset of $K$ of the form $$B_{\gamma}(a):=\{x\in K:v(x-a)>\gamma\}$$ where $a\in K$ and $\gamma\in \Gamma$.
A closed ball is a subset of $K$ of the form
$$B_{\geq\gamma}(a)=\{x\in K:v(x-a)\geq\gamma\}$$

\end{defi}

In this setting we have two natural topologies:  the Zariski and the valuation topologies. The latter is generated by the balls. When we say `open' or `closed' we mean in the valuation topology. When we want to refer to the Zariski topology we will be explicit about it.

\begin{defi}
If $D\subseteq K^n$ is a $\mathbb K$-definable set we define $\dim D$ as the usual algebraic dimension of the Zariski closure of $D$. Moreover if $p(x)$ is an $n$-type over $A$, we define $\dim p=\min\{\dim X:X\in p\}$ and if $a\in K^n$ then $\dim (a/A)$ is defined as $\dim (tp(a/A))$
\end{defi}

\begin{defi}
If $X\subseteq K^n$ is $\mathbb K$-definable over $A$, we say that $x=(x_1\ldots,x_m)\in X^m$ is a tuple of independent generics of $X$ if $\dim (x/A)=m\dim (X)$
\end{defi}

We have the next classic theorem that follows from Holly's work in \cite{HO}. As stated here is Theorem 7.1 of \cite{hru-haskell2005stable}.

\begin{fact}\label{QEf} (Quantifier Elimination)
$K$ has Quantifier elimination in the language $\mathcal L_{R,v}$. In particular any definable subset of $K^n$ is the intersection of a (valued) open subset of $K^n$ with a Zariski closed set. 
\end{fact}

As a corollaries we have:
\begin{fact}\label{decompositionLemma}
Suppose $Y\s K^2$ is $\mathbb{K}$-definable and infinite. Then $Y$ can be written as a finite union of  subsets of $K^2$ that are relatively open subsets of irreducible Zariski closed sets.
\end{fact}

\begin{fact}
If $a\in \acl_{\mathbb K} (B)$ for some set of parameters $B$ then there is a finite set definable just with the field sructure of $\mathbb K$ with parameters $B$ containing $a$.

In other words $\acl_{\mathbb K}=\acl_{\mathbb K^{f}}$ where $\mathbb K^{f}$ is $\mathbb K$ seen as a structure in the language of rings.
\end{fact}

Now we fix some concepts:
\begin{defi}
If $U \subseteq K^n$ is an open set and $f:U\to K$ is a function, given $a\in U$ we say that $f$ is analytic at $a$ if there is some power series $$g=\sum a_I x^I$$ converging in a neighborhood $U'$ of $0$ such that $f(z+a)=g(z)$ for every $z$ in $U'$. 

We say that $f$ is analytic in $U$ if it is analytic at $a$ for every $a\in U$
\end{defi}

\begin{defi}
If $$F(x_1\ldots,x_n)=\sum_{I=(i_1,\ldots,i_n)}a_I x^{(i_1,\ldots,t_n)},$$  is a power series then $F_{x_k}(x_1\ldots,x_n)$ is the (formal) partial derivative of $F$ respect to $x_k$, that is $$F_{x_k}(x_i,\ldots,x_n)=\sum_{I=(i_1,\ldots,i_n)} i_k a_I  x^{(i_1,\ldots,i_k-1,\ldots,i_n)}.$$
\end{defi}

\begin{defi}
Let $U\subseteq K$ be an open set and let $a\in U$. Suppose $f:U\to K$ analytic at $a$ and its expansion around $a$ is $$f(x)=\sum_{n\geq 0} b_n x^n.$$ We say that $a$ is a zero of $f$ with multiplicity $d$ if $f(a)=0$ and $d=\min\{n:b_n\neq 0\}$.
\end{defi}

\begin{lemma}
Let $f:U\to V$ be a $\mathbb K$ definable function with $U$ and $V$ open subsets of $K^n$ and $K^m$ respectively, then there is some $a\in U$ such that $f$ is analytic at $a$.
\end{lemma}

\begin{proof}

\end{proof}

\begin{comment}

\end{comment}

We will also need an implicit function theorem, the following is well known and can be found for example in \cite{LAG} (Theorem (10.8), page 84)

\begin{fact}(Implicit Function Theorem)\label{implicitFunctionTheoremf}
Suppose $K$ is complete, $U\subseteq K^n$ is open and $F:U\to K$ is an analytic function at $z\in U$. Assume $z$ is such that $F_{x_n}(z)\neq 0$, then there are $U_1\subseteq K^{n-1}$ and $U_2$ open subsets of $K$, $U'\subseteq U$ open with $z\in U' \cap (U_1\times U_2)$ and $f:U_1\to U_2$ analytic such that:

\begin{equation*}
    \{u\in U':F(u)=0\}=\{(x,f(x)):x\in U_1\}.
\end{equation*}
\end{fact}

As a corollary of this we have the Inverse Function Theorem:

\begin{fact}\label{inversionf}(Inverse Function Theorem, Theorem 10.10 in \cite{LAG})
Let $F(x)$ be analytic at $0$, assume $F(0)=0$ and $F_x (0)\neq 0$ then there is an unique analytic function $G$ converging in a neighborhood of $0$ such that $F(G(y))=y$ for each $y$ in such a neighborhood.
\end{fact}

Now we state some strong results on valued fields, first a theorem about continuity of roots:

\begin{fact}\label{continuityOfRootsTheoremf}(Continuity of Roots)

Assume $K$ is complete. Suppose $U_1, U_2$ are open subsets of $K$ and $$F:U_1\times U_2\to K$$ is analytic  at $(a,b)\in U_1\times U_2$. Assume that the function $F(*,b)$ has a zero of multiplicity $d>0$ at $a$. Then there are open sets $U_a$ and $U_b$ with 
$a\in U_a\subseteq U_1$ and
$b\in U_b\subseteq U_2$ such that:

\begin{enumerate}
    \item $a$ is the only $x\in U_a$ such that $F(x,b)=0$
    \item For each $y\in U_b$ the function $F(*,b)$ has exactly $d$ zeros in $U_a$ (counting multiplicities)
\end{enumerate}

\end{fact}

\begin{proof}
We may assume $(a,b)=(0,0)$. By Theorem (10.3)(2) of \cite{LAG} there is some $F^*(x,y)=f_0(y)+f_1(y)x+\ldots+x^d$ such that each $f_j$ is analytic, $f_i(0)=0$ for $i=1,2,\ldots, d-1$, and there is some $\delta$ an unit of $K[[x,y]]$ such that $F=F^* \delta$. Because of Theorem (11.3) of \cite{LAG} there is an open set $U_b$ such that each $f_i$ is convergent in $U_b$, $0\in U_b$ and for all $y\in K$ if $y\in U_b$ and $F^*(x,y)=0$ then $x\in U_1$. Moreover for all $y\in U_b$ the polynomial $F^*(*,y)$ is a polynomial of degree $d$ in the first variable so as $K$ is algebraically closed there are exactly $d$ roots (counting multiplicities) and all of them belongs to $U_1$ 
\end{proof}

We will also need an identity theorem for expansion of power series:

\begin{fact}\label{identityTheoremf} (Identity Theorem (10.5.2) of \cite{LAG})

Let $$f(x_1,\ldots,x_n)=\sum_I a_I x_1^{i_1}\ldots x_x^{i_n}$$ be a power series converging in a neighborhood $D(f)$ of $(0,\ldots,0)$. Assume that for all $x$ in some open set $U\subseteq D(f)$ one has that $f(x)=0$, then $a_I=0$ for all $I$.
\end{fact}

\section{Partial Isomorphism}\label{partialIsomorphism}

In this section we present the proof of Theorem \ref{confiGroupIso}, we follow lines on  \cite{HRPIGroups}.

From now let $\mathcal H=(H,\otimes)$ be a fixed group definable in $\mathbb K$ with parameters $B$.% We assume $\mathcal H$ is strongly minimal. 

\begin{thm}\label{confiGroupIso}
Let $(G,\oplus)$ be a group definable in the structure $\mathbb K$ with parameters $B$. Assume $a,b\in G$ are such that $\dim_{\mathbb K}(a/B)=\dim_{\mathbb K}(b/B)=\dim(G)$. And $\dim_{\mathbb K}(ab/B)=2\dim(G)$.  Let $c=a\oplus b$. Assume moreover that there are $a', b' \in H$ with $\dim_{\mathbb K}(a'/B)=\dim_{\mathbb K}(b'/B)=\dim H$ and $\dim_{\mathbb K}(a'b'/B)=2\dim(H)$. Assume as well that $\dcl_{\mathbb K}(aB)=\dcl_{\mathbb K}(a'B)$, $\dcl_{\mathbb K}(bB)=\dcl_{\mathbb K}(b'B)$ and $\dcl_{\mathbb K}(cB)=\dcl_{\mathbb K}(c'B)$ where $c'=a'\otimes b'$. Then there is a $\mathbb K$-definable analytic,  local isomorphism between a neighborhood of the identity of $G$ and a neighborhood of the identity of $H$. That is, there is a neighborhood $U_1$ of the identity of $G$, a neighborhood $U_2$ of the identity of $H$ and an analytic invertible function $\phi:U_1\to U_2$ whose inverse is analytic such that if $u,v\in U_2$ are such that $u\oplus v\in U_1$, then $\phi(u\oplus v)=\phi(u)\otimes \phi(v)$.
\end{thm}

\begin{proof}
Our assumptions on $a$ and $a'$ implies that there is a formula $\phi(x,y)$ (with parameters in $B$) such that $\phi(a,a')$ holds and there is just one $x$ such that $\phi(x,a')$ holds and just one $y$ such that $\phi(a,y)$ holds. Our dimension assumptions on $a$ and $a'$ implies that there is an open neighborhood $U_1$ of $a$ and $U_2$ of $a'$ such that for all $x\in U_1$ there is just one $y$ such that $\phi(x,y)$ holds and for all $y\in U_2$ there is just one $x$ such that $\phi(x,y)$ holds. Therefore $\phi(x,y)$ defines the graph of a function $f:U_1\to H$ and the opposite defines a function $f^{-1}:U_2\to G$. 
%Using \ref{QE2} there is an open dense subset $U_3\subseteq U_1$ ($B$-definable) such that the restriction of $f$ to $U_3$ is an analytic function. Similarly we find $U_4\subseteq U_2$ also $B$-definable, such that the restriction of $f^{-1}$ to $U_4$ is analytic. 

Using again our dimension assumptions on $a$ and $a'$ we have that $a\in U_1$ and $a'\in U_2$. So we get open sets $U$ and $U'$ with $a\in U$ and $a'\in U'$  and an analytic invertible function $f:U\to U'$ with analytic inverse such that $f(a)=a'$. 

We can do the same thing for $b,b'$ and $c,c'$ finding open neighborhoods  $V$, $V'$, $W$ and $W'$ of $b,b',c$ and $c'$ respectivelly and $\mathbb K$-definable invertible functions $g:V\to V'$ and $h:W\to W'$ such that $g(b)=b'$ and $h(c)=c'$. 

Now as $f(a)\otimes g(b)=h(c)$ and $\dim(ab/B)=2\dim G$ there is an open neighborhood $Z$ of $(a,b)$ in $G^2$ such that for all $(x,y)\in Z$ one has that $f(x)\otimes g(y)=h(x\oplus y)$. Moreover we may shrink $U$ and $V$ to ensure that $U\times V\subseteq Z$. So for all $x\in U$ and $y\in V$ one has that $f(x)\otimes g(y)=h(x\oplus y)$. We may also shrink $U$ an $V$ to ensure that $U\oplus V\subseteq W$

Let $x,y$ and $z$ be elements of $U$ such that $a=x\oplus z^{-1}\oplus y$ we claim that $f(a)=f(x)\otimes f(z)^{-1}\otimes f(y)$, for this let $b_1=x^{-1}\oplus c$ and $c_1=z\oplus b_1$. Then as $a$ and $b$ are independent over $B$ one has that $b_1\in V$ and $c_1\in W$. Moreover $y\oplus b=c_1$. So because of previous observation one has:

$f(y)\otimes g(b)=h(c_1)$

$f(x)\otimes g(b_1)=h(c_1)$

$f(x)\otimes g(b_1)=h(c)$

Thus $f(x)\otimes f(z)^{-1}\otimes f(y)\otimes g(b)=f(x)\otimes g(b_1)=h(c)=f(a)\otimes g(b)$ and we get our claim.

Now define $\phi:U^{-1}\oplus a\to (U')^{-1}\otimes a'$ as $\phi(x^{-1}\oplus a)=(f(x))^{-1}\otimes a'$ we claim that it is a local isomorphism between $U_1:=U^{-1}\oplus a$ and $U_1':=(U')^{-1}\otimes a'$. For this let $u,v\in U_1$ such that $u\oplus v\in U_1$ we want to prove that $\phi(u\oplus v)=\phi(u)\otimes \phi(v)$. Suppose $u=x^{-1}\oplus a$ and $v=y^{-1}\oplus a$ with $x,y\in U$. Therefore $u\oplus v=z^{-1}\oplus a$ where $z^{-1}=x^{-1}\oplus a\oplus y^{-1}$ with $z\in U$. So using the claim just proved, we have that $f(z)^{-1}=f(x)^{-1}\otimes f(a)\otimes f(y)^{-1}$ and then $\phi(z^{-1}\oplus a)=f(z)^{-1}\otimes f(a)=f(x)^{-1}\otimes f(a)\otimes f(y)^{-1}\otimes f(a)=\phi(x^{-1}\oplus a)\otimes \phi(y^{-1}\oplus a)$ and we get the result.
\end{proof}

\chapter{Some facts of definable groups contained in $K$}\label{someFacts}

In this chapter we give some basics on $\mathbb K$-definable groups whose universe is contained in $K$. We also prove the main tool we will use in order to construct inerpretable fields, namely Theorem \ref{thmFieldGivenFamily}.

In the first section present some reductions that will be use in the second one. In the second section we prove Theorem \ref{thmFieldGivenFamily} that ensures the existence of an interpretable field under the assumption of some kind of families of curves. 

For the rest of the chapter we fix $(G,\oplus)$ a $\mathbb K$-definable group of dimension one and we assume $G\subseteq K$. Let $e$ be the neutral element of $G$. We assume that the map $(x,y)\mapsto x\oplus y^{-1}$ is continuous.

\begin{defi}
If $X\subseteq G^2$ is a $\mathbb K$-definable set we say that $X$ is $G$-affine if it is a boolean combination of cosets of $\mathbb K$-definable subgroups of $G^2$.
\end{defi}

We fix $X\subseteq G^2$ a non affine curve and assume that $\mathcal G=(G,+,X)$ is strongly minimal.

\section{Some reductions}\label{somePreparations}

%Note that as $G$ is infinite, it contains an open ball, so there is $B\subseteq G$ some open ball say around $s$. Then given any $r\in G$ as the map $x\maps to $ 

In this section we give some basic definitions and reductions that we use in Chapter \ref{additiveCase} and \ref{multiplicativeCase}. The main result is Proposition \ref{assmGraphpC}. We would like to treat $X$ (generically) as the graph of a function using Fact \ref{implicitFunctionTheoremf} but in order to do so we need the following:

\begin{lemma}\label{pickCoordinateLemmaP}
 Let $F(x,y)\in K[x,y]$ be an irreducible polynomial and suppose that $F$ is not constant. If there are infinitely many points $(a,b)\in K^2$ such that 
 \begin{equation}
     F(a,b)= F_{y}(a,b)=0,
 \end{equation}
 then $F(x,y)=G(x,y^{p})$ for some polynomial $G$.
\end{lemma}

\begin{proof}

Note first that $\deg(F_y)<\deg(F)$. So as $F$ is irreducible, if $F_y\neq 0$ we can use B\'ezout's Theorem to conclude that the number of common zeros of $F$ and $F_y$ equals $\deg F\cdot \deg F_y$. So if $F_y$ has infinitely many common zeros with $F$ it is because $F_y=0$.

 If char$(K)=0$ and $F_y=0$ using Fact \ref{identityTheoremf} we can see that $F(x,y)=kx+l=F(x,y^p)$ with $k,l\in K$.
 
So assume char $(K)=p>0$. Let $V\subseteq K^2$ the set of zeros for $F$. 

%It is clear that $\deg (F_y)\leq \deg(F)$ and since $F$ is irreducible it follows that $F$ and  $F_y$ have no common factors. So if $F_y\neq 0$ we can use B\'ezout's Theorem to conclude that the number of common zeros of $F$ and $F_y$ equals $\deg F\cdot \deg F_y$. So i

Assume $F_y$ has infinitely many zeros at $V$ so $F_y=0$. Write $$F(x,y)=f_0(x)+f_1(x) y +\ldots+f_n(x) y^n$$ with $f_i(x)\in K[x]$ for $i=0,1,\ldots, n$.  $$F_y(x,y)=f_1(x)+2f_2(x)y+\ldots+nf_n(x)y^{n-1}=0,$$ so for any $a$ and any $i=1,\ldots, n$ we have that $if_i(a)=0$. If some $f_i$ is different from zero we can take $a\in K$ such that $f_i(a)\neq 0$ and $if_i(a)=0$ implies that $i=0$ so $i=mp$ for some $m$. Therefore if $f_i$ is different from zero implies $p\mid i$ so $F(x,y)$ is a polynomial in the variable $y^p$.   
 \end{proof}
 
 As an immediate corolary we have:
 
 \begin{lemma}\label{pickCoordinateLemma}
 Assume $F(x,y)$ is an irreducible polynomial and let $V$ the set of common zeros of $F$ then there is some finite subset of $V$, $E$ such that either for all $a\in V\setminus E$, $F_x(a)\neq 0$ or for all $a\in V\setminus E$, $F_y(a)\neq 0$.
 
 \end{lemma} 

\begin{proof}
Assume it is not the case, then, both $F_x$ and $F_y$ have infinitelly many zeros in $V$. So using Lemma \ref{pickCoordinateLemmaP} one can conclude that $F(x,y)=G(x^p,y^p)$ for some polynomial $G$ but then $F(x,y)=(\tilde G(x,y))^p$ for some polynomial $\tilde G$ and then $F$ is not irreducible.
\end{proof}

\begin{lemma}\label{assumptionLemma}

 There is a finite set $F\subseteq X$ such that for any $z=(z_1,z_2)\in X\setminus F$ either there is an open set $U'$ with $z_1\in U'$ and an analytic function $f:U'\to K$ such that $(x,f(x))\in X$ for all $x\in U'$ or there is $U''$ open with $z_2\in U''$ and an analytic function $g:U''\to K$ such that $(g(x),x)\in X$ for all $x\in U''.$
 \end{lemma}

\begin{proof}
Decompose $X=X_0\cup\ldots\cup X_n$ as in Fact \ref{decompositionLemma}  with $X_0$ finite and each $X_i$ in an open set of some Zarizki closed set $C_i$ for $i>0$. By adding the intersection points to $X_0$ we can take the union to be disjoint. As $C_i$ is the set of zeros of some irreducible polynomial $F_i$ we can apply Lemma \ref{pickCoordinateLemma} to each $F_i$,  either $ (F_i)_{x}$ or $(F_i)_{y}$ has just finitely many zeros at $C_i$. Call $E_i$ this finite set intersected with $X_i$ and put $E=X_0\cup E_1\cup \ldots\cup E_n$. If $z\in X\setminus F$ there is just one $i$ such that $z\in X_i$, and either $(F_i)_{x}(z)$ or $(F_i)_{y}(z)$ is different from zero. Suppose $(F_i)_{y}(z)\neq 0$ by \ref{implicitFunctionTheoremf} we get an open set $U$ with $z_1\in U$ and an analytic function $f:U'\to K$ such that $\{(x,f(x)):x\in U'\}\subseteq C_i$ as $X_i$ is an open subset of $C_i$. Replace  $U'$ by $U'\cap f^{-1}(X_i)$ if necessary, we get $\{(x,f(x)):x\in U'\}\subseteq X_i$. This proves the lemma.
\end{proof}

\begin{defi}
If $W\subseteq K^2$ is a one dimensional $\mathbb K$-definable set we can decompose $W=W_0\cup W_1\cup\ldots\cup W_n$ where $W_0$ is finite and for $i\geq 1$, $W_i$ is an open subset of some irreducible Zariski closed set $\tilde W_i$. If $a\in W_i$ we say that $\tilde W_i$ is a branch of $W$ at $a$. 

For $i>0$ suppose $\tilde W_i$ is the set of zeros of an irreducible polynomial $F_i$, then  

\[
\begin{array}{ll}
W^2:=& \{a\in W:a\in W_i,\ (F_i)_y(a)\neq 0,\ 1\leq i\leq n \},\text{ and }\\
W^1:=& \{a\in W:a\in W_i,\ (F_i)_x(a)\neq 0,\ 1\leq i\leq n \}.
\end{array}
\]

\end{defi}

\begin{defi}\label{defSlope}
Let $Z\subseteq K\times K$ be the set of zeros of some irreducible polynomial $f(x,y)$. If  $a\in Z$ we define: $$m_1(a,Z)=-f_y(a)/f_x(a)$$ and $$m_2(a,Z)=-f_x(a)/f_y(a).$$  We say that $m_2(a,Z)$ is the slope of $Z$ at $a$ and $m_1(a,Z)$ is the inverse slope of $Z$ at $a$ when they are both well defined.
\end{defi}

%%In particular, Lemma above defines a function $U\cap X\to K$ where $U$ is some open subset of $G\times G$ containing $(0,0)$ let's denote this function by $X(z)$, since this is a locally analytic function, it has a derivative, denote this by $X'(z)$. Call $X_D$ to the domain of this function. That is, $X_D$ is some open subset of $X$ with $(0,0)\in X_D$ and for each $z\in X_D$ there is a partial analytic function $f$ and an open set $U\subseteq G\times G$ such that $z\in U$ and $X\cap U$ is the graph of $f$. 

\begin{lemma}\label{lemmaSlope}
 Let $W\subseteq K\times K$ be a $\mathbb K$-definable set whose Zariski closure has algebraic dimension $1$. If $a=(a_1,a_2)\in W^2$ there is $U'$ open neighborhood of $a_1$ and an analytic function $f:U'\to K$ such that $f(a_1)=a_2$, $\{(x,f(x)):x\in U'\}\subseteq W$ and $m_2(a,\tilde W_i)=f'(a_1)$. Where $W_i$ is the only component such that $a\in W_i$.
\end{lemma}

\begin{proof}
Suppose $\tilde W_i$ is the zero set of $G_i(x,y)$, so $(G_i)_y(a)\neq 0$ so using Fact \ref{implicitFunctionTheoremf} there are open sets $U_1$ and $U_2$ with $a\in U_1\times U_2$; and there is also a function $f:U_1\to U_2$ analytic at $a$ and an open set $U'\subseteq U_1\times U_2$ such that $$\Gamma:=\{u\in U':G_i(u)=0\}=\{(x,f(x)): x\in U_1\};$$ as $W\cap \tilde W_i$ is open in $\tilde W_i$ we can assume $\Gamma\subseteq W$. Now as $G_i(x,f(x))=0$ in a neighborhood of $a_1$, if we derive with respect to $x$ the function $G_i(x,f(x))$ we get $(G_i)_x+ (G_i)_y f'(x)=0$ so $f'(a_1)=-(G_i)_x(a)/(G_i)_y(a)=m_2(a,\tilde W_i)$.
\end{proof}

\begin{prop}\label{assmGraphpC}

There is a set $Y\subseteq G\times G$ definable in $(G,\oplus,X)$ such that: The $\dim Y=1$ and for each $(x_0,y_0)\in Y$ there is an analytic function $g$ defined in a neighborhood $U_0$ of $x_0$ and some natural number $n$ such that $\text{Fr}^{-n}(g(x_0))=y_0$ and $$\left(z,\text{Fr}^{-n}\left(g(z)\right)\right)\in Y$$ for each $z\in U_0$. Moreover, for $x_0=e$ we can take $n=0$; that is, there is an analytic function $h$ defined in an open neighborhood $U$ of $e$ such that $(x,h(x))\in Y$ for all $x\in U$.

\end{prop}
 %In the same way if $a\in Y_1$ there is $U''$ an open neighborhood of $a_2$ and an analytic function $g:U''\to K$ such that $(g(x),x)\in Y$ for all $x\in U''$. In this case $m_1(a)=g'(a_2,V_i)$
 \begin{proof}
 
 Let $F\subseteq X$ be as in Lemma \ref{assumptionLemma}. Let $x=(x_1,x_2)\in X\setminus F$ so by the conclusion either $F\setminus X$ or $(F\setminus X)^{-1}$ contains the graph of a function. Define $Y'$ as $X\setminus F$ if it contains the graph of an analytic function converging in a neighborhood of $x_1$ or as $(X\setminus F)^{-1}$ if it doesn't. Define $y'=(x_1,x_2)$ if $Y'=F\setminus X$ and $y'=(x_2,x_1)$ if $Y'=(F\setminus X)^{-1}$. So if $y'=(y_1,y_2)$ there is an analytic function converging in a neighborhood of $y_1$ whose graph is contained in $Y'$.
 
 Now let $Y''=Y'_{y'}$, so there is an analytic function $h$ converging in a neighborhood of $e$ whose graph is contained in $Y''$.

 Write $Y''=Y''_0\cup \cdots \cup Y''_n$ and let $G_i$ be a polynomial whose set of zeros is the Zariski closure of $Y''_i$. 
 
 If $(G_i)_y(x,y)$ has just finite zeros at $Y''_i$ then define $E_i$ as this set of zeros.
 
  If $(G_i)_{y}$ is zero for infinitely many points of $Y''_i$,  then, by Lemma \ref{pickCoordinateLemmaP} $G_i(x,y)$ is a function in the variable $y^{p^n}$ for some $n$, that is $G_i(x,y)=G(x,y^{p^n})$ for some polynomial $G$ and some natural number $n$. Let $k$ be the maximum among those possible $n$ and let $G$ be the corresponding polynomial. Therefore $G(x,y)$ is not a function in the variable $y^p$ so the derivative $G_y$ has just finitely many zeros in $Y_i$. Let $E_i$ this finite set. Then for each $(x_0,y_0)\in Y_i\setminus E_i$ there is an analytic function $g$ defined in some neighborhood of $x_0$ such that $G_i(z,g(z))=0$ for each $z$ in that neighborhood. So if we define $\text{Fr}^k:K\to K$ as $\text{Fr}^k(x)=x^{p^k}$ and $\text{Fr}^{-k}$ it's inverse, $F_i(z, \text{Fr}^{-k}(g(z)))=0$ so $Y''_i$ contains the graph of $z\mapsto \text{Fr}^{-k}(g(z))$ in some neighborhood of $x_0$.
  
  Therefore if we define $E=Y''_0\cup E_1\cup \ldots E_n $ then $Y=Y''\setminus E$ satisfies conclusion of proposition.

%Then, because of Lemma \ref{pickCoordinateLemma} there is some point $(x_0,y_0)\in Y_i$, an open set $U_0$ containing $y_0$ and an analytic function $f_0:U_0\to K$ such that $\{(f_0(y'),y'):y'\in U_0\}\subseteq Y_i$. Moreover $\partial F_i/\partial y(f_0(y'),y')=0$ for all $y'\in U_0$ but then using Theorem \ref{identityTheorem} one gets that $\partial F_i/\partial y(f_0(y'),y')$ is the constant function $0$ and therefore if $$F_i(x,y)=\sum_I a_I x^{i_1} y^{i_2}$$ then $$\partial F_i/\partial y (x,y)=\sum_I a_I i_2 x^{i_1} y^{i_2-1}$$ as we may assume $f_0$ is not the zero function there is an open set where $f_0(y)\neq 0$ and $$\sum_I a_I i_2 (f_0(y))^{i_1} y^{i_2-1}=0.$$ 
%Using Theorem \ref{identityTheorem} this implies $a_I i_1=0$ so if $a_I\neq 0$ then $i_2=0$, that is $p\mid i_2$
\end{proof}

For the rest of this section we fix the following notation:

Let $Y$ be as provided by Proposition \ref{assmGraphpC}. Decompose $Y=Y_1\cup\ldots\cup Y_n$ with $Y_i$ some open subset of a Zariski closed set $C_i$. And let $F_i$ be an irreducible polynomial whose zero set is $C_i$.

\begin{defi}
If $a\in Y^2$ let $Y'(a)=m_2(a,C_i)$ where $C_i$ is the only branch of $Y$ at $a$.
\end{defi}

The following is Fact 3.8 of \cite{KR}, although the proof is just to apply usual $\epsilon-\delta$ formula.
\begin{fact}\label{derivativeDefinible} 
Let $h$ be as in Proposition \ref{assmGraphpC}, then the function $z\mapsto h'(z)$ with domain $U$ is definable in $\mathbb{K}$.
\end{fact}

\begin{defi}\label{AdditionAndCompositionDef}
If $V,W$ are subsets of $G^2$ and $a=(a_1,a_2)\in V$ we define:

\begin{itemize}
    \item $V\oplus W=\{(x,y\oplus z):(x,y)\in V \text{ and } (x,z)\in W\}$
    \item $V\setminus W=\{(x,y\cdot z^{-1}):(x,y)\in V\text{ and }(x,z)\in W\}$
    \item $V\circ W =\{(x,y):\exists z  (x,z)\in W\wedge (z,y)\in V\}$
    \item $V_a=\{(v_1\oplus a_1^{-1},v_2\oplus a_2^{-1}):(v_1,v_2)\in V\}$
\end{itemize}
\end{defi}

For the rest of the Chapter we will make the following assumption on $G$:

\begin{assm}\label{assmOperationSumDerivative}
 For each $U\subseteq K$ open set with $e\in U\subseteq G$ and each pair of analytic functions $f,g:U\to K$ with $f(U)\cup g(U)\subseteq G$ such that $f(e)=g(e)=e$, if we define $h(x)=f(x)\cdot g(x)$ then $h$ is analytic and $h'(e)=f'(e) + g'(e)$. And if we define $i(x)=(f(x))^{-1}$ (the inverse in $G$) then $i'(e)=-f'(e)$
\end{assm}

Note that if $G=(B,+)$ is a subgroup of $(K,+)$ or $G=(K\setminus\{0\},\cdot)$, then $G$ satisfies Assumption \ref{assmOperationSumDerivative}.

\begin{lemma}\label{additionAndCompositionLemma} 
%Let $V$ and $W$ be $\mathbb K$-definable subsets of $K^2$  defining analytic functions, denote it's derivative by $z\mapsto V'(z)$ and $z\mapsto W'(z)$ just as before. Call $V_D$ and $W_D$ to their respective domains. Assume $(0,0)\in V_D\cap W_D$. Then $-V$, $V+W$, and $V\circ W$ also define locally analytic functions with domain some neighborhood of $(0,0)$

If $a,b\in Y^2$ there is a branch $Z$ of $Y_a\oplus Y_b$ at $(e,e)$ such that $$m_2((e,e),Z)=Y'(a)+Y'(b).$$ There is also a branch $\hat Z$ of $Y_a\circ Y_b$ at $(e,e)$ such that $$m_2((e,e),\hat Z)= Y'(a) Y'(b).$$

\end{lemma}

\begin{proof}
Write $a=(a_1,a_2)$ and $b=(b_1,b_2)$. 
Let $h:U\to K$ be an analytic function whose graph is contained in $Y$ with $a_1,b_1\in U$. Then the graph of the function $h_a(x):=h(x\oplus a_1)\oplus a^{-1}_2$ is contained in $Y_a$. Similarly the graph of $h_b(x):=h(x\oplus b_1)\oplus b^{-1}_2$ is contained in $Y_b$. So the graph of $x\mapsto h_a(x)\oplus h_b(x)$ is contained in $Y_a\oplus Y_b$ and we can take $Z$ as the Zariski closure of such graph and use Assumption \ref{assmOperationSumDerivative}. In the same way, the graph of $h_a\circ h_b$ is contained in $Y_a\circ Y_b$.
\end{proof}
We can therefore define:
\begin{defi}
For $a,b\in Y^{2}$ let $$(Y_a\oplus Y_b)'((e,e))=m_2((e,e),Z)$$ and
$$(Y_a\circ Y_b)'((e,e))=m_2((e,e),\hat Z)$$ where $Z$ and $\hat Z$ are as in Lemma \ref{additionAndCompositionLemma}
\end{defi}

%\begin{defi}\label{function}
 %As $(0,0)\in Y_2$ there is an open set $U\subseteq G$ with $0\in U$ and an analytic function $h:U\to K$ such that $(x,h(x))\in Y$ for each $x\in U$. Therefore for each $a=(a_1,a_2)$ in some open subset of $Y$ containing $(0,0)$ we can define $Y'(a):=h'(a_1)$.
 %\end{defi}

 We will also need a technical lemma:
 
 \begin{lemma}\label{noIsolated}
 Let $V$ and $W$ be one dimensional $\mathbb K$-definable subsets of $G^2$. Suppose moreover that they do not have isolated points, then $V\oplus W$  does not have isolated points. If, moreover, we assume that $W$ does not contain infinitely many points in any horizontal line and $V$ does not contain infinitely many points in any vertical line, then the composition $V\circ W$ does not have isolated points.
 \end{lemma}
 
 \begin{proof}
 %Let's prove first that $Y_a + Y_b$ don't have any isolated point. It is clear that if $a=(a_1,a_2)$ and $x=(x_1,x_2)\in Y_a$ is isolated then $(x_1+a_1,x_2+a_2)$ is isolated in $Y$ so $Y_a$ and $Y_b$ don't have isolated points.
 
 If $(a,b)\in V\oplus W$ there are $z$, $w$ such that $(a,c)\in V$, $(a,d)\in W$ and $b=c\oplus d$. Take an open sets $B_1$ containing $a$ and $B_2$ containing $b$. As $\oplus$ is continuous one can find $D$ and $E$ open sets containing $c$ and $d$ respectively such that $D\oplus E \subseteq B_2$. 
 
 Given that $(a,c)\in V$ is not isolated there is a polynomial $F$ such that $F(a,c)=0$ and for all $(a',c')$ in some neighborhood of $(a,c)$ if $F(a',c')=0$ then $(a',c')\in V$. If the polynomial $F_a(x):=F(a,x)$ is not zero we can use continuity of roots and get $B$, a neighborhood of $a$ (we can assume $B\subseteq B_1$) such that and for each $a'\in B$ there is $c'\in D$ with $F(a',c')=0$ and $(a',c')\in V$. In the same way (by restricting $B$ if necessary) for each $a'\in B$ there is $d'\in E$ such that $(a',d')\in W$ so for any such $a'$ the point $(a',c'\oplus d')\in (V\oplus W)\cap (B_1\times B_2)$ and then $(a,b)$ is not isolated in $V\oplus W$. 
 %$\overline{Y} \bar Y$  $\tilde Y$ $\widetilde Y$  $\mathbb N$
 
 Note that if one of $F(a,\_)$ is the zero polynomial then $V$ contains infinitely many points in the vertical line $\{a\}\times K$, and so with $V\oplus W$, so that any open containing $(a,b)$ contains points in such a line. Similarly for $W$.
 
We will now prove that $V\circ W$ doesn't contain any isolated point. Let $(a,b)\in V\circ W$ so there is $c$ with $(a,c)\in V$ and $(c,b)\in W$.
 If $B_1$ and $B_2$ are open sets containing $a$ and $b$ respectively, as in the addition one can find $B$ containing $c$ such that for each $c'\in B$ there is $a'\in B_1$ such that $(a',c')\in Y_b$ (using that the closure of $W$ doesn't contain any horizontal line) and some $b'\in B_2$ such that $(c',b')\in V$ (using that the closure of $V$ doesn't contain any vertical line) so that for each such $c'$ the point $(a',b')\in (V\circ W)\cap(B_1\times B_2)$ and $(a,b)$ is not an isolated point of $V\circ W$.
 \end{proof}
 
% Let $Y=X\setminus F$ as before. We can assume $(e,e)$ is a point in the relative interior of $Y_2$ so there is an analytic function $h:U\to K$ such that $e\in U$ and $(x,h(x))\in Y$ for all $x\in U$. 
% Let $t=h'(e)$.
 
\section{Good families of curves}\label{goodFamiliesOfCurves}
In this section we define the notion of a good family of curves for $\mathcal G$ and we show how it can be used to interpret a field in $\mathcal G$. This is done in Theorem \ref{thmFieldGivenFamily}.

\begin{defi}\label{defGoodFamily}
A good family of curves definable in $\mathcal G$ is given by the following data:
\begin{enumerate}
    \item A $\mathcal G$-definable set $Y\subseteq G\times G$ and a $\mathcal G$-definable family of curves $(X_a)_{a\in Y}$.
    \item $U\subseteq G$, an open neighborhood of $e$, and an analytic function $H(x,s):U\times U\to G$.
    \item An analytic function $h:U\to G$.
\end{enumerate}
Such that:
\begin{enumerate}
    
    \item For all $s\in U$ if we define $\bar s:=(s,h(s))$ then $\bar s\in Y$.
    
    %\item For all $(x,y)\in Y$ there is an analytic function $g$ defined in a neighborhood $U_x$ of $x$ and a natural number $n$ such that for all $z\in U_x$ if we define $\bar z:=(z,\text{Fr}^{-n}(g(z)))$ then $\bar z\in Y$
    
    \item Each $X_a$ is a one dimensional subset of $K^2$ with no isolated points.
    \item For all $s,x\in U$, $(x,H(x,s))\in X_{\bar s}$.
    
    \item $H(e,s)=e$ for all $s\in U$.
    \item The set of derivatives $\displaystyle\{H_x (e,s):s\in U\}$ is infinite.
    
    \item For each $s\in U$ and $(x,y)\in X_{\bar s}$ there is a neighborhood $V_x$ of $x$ and a neighborhood $U_s\subseteq U$ of $s$, an analytic function $\Phi(z,\delta)$ defined in $V_x\times U_s$ and a natural number $n$ such that for all $(w,\delta)\in V_x\times U_s$ one has that $(w,\text{Fr}^{-n}(\Phi(w,\delta)))\in X_{\bar \delta}$.
\end{enumerate}
\end{defi}

 We will dedicate the rest of the section to proving:
 
 \begin{thm}\label{thmFieldGivenFamily}
If there is a good family of curves definable in $\mathcal G$ then $\mathcal G$ interprets a field. 
\end{thm}
 
So for  the rest of this section we fix $Y$, $(X_a)_{a\in Y}$, $U$, $H$ and $h$ as in the data of a good family of curves.

\begin{lemma}\label{tupla5dim}
There exists a ball $B\subseteq K$ with center $t$ contained in $\{H_x(e,s):s\in U\}$ and $\mathfrak a=(a_1,a_2,b_1,b_2,b)\in K^5$ with $\dim \mathfrak a=5$ such that $$\{t a_1, \ t+a_2 , \ tb_1,\  t+b_2,\  t a_1 a_2 ,\  t+a_1b_2+a_2,\  t+b, \ t+a_1 b+a_2,\  t+a_1 b_1 b + b_1 a_2 +b_2\}\subseteq B$$.

\end{lemma}
\begin{proof}
As $\{H_x(e,s):s\in U\}$ is infinite and $\mathbb K$-definible, it contains an open ball say $B\subseteq G$ with $0\neq t=H_x(e,d)\in B$ for some $d\in U$.
Then  $\mathfrak a\in K^5$ exists because $B$ is a non empty open and both addition and multiplication are continuous functions.

\end{proof}

For $s\in U$ define $Y'(s)=H_x(e,s)$. 
So we have:

\begin{lemma}\label{defiDerivadaConf}

Let $\mathfrak a$ as in Lemma \ref{tupla5dim}, then there are tuples:
$\alpha=(\alpha_1,\alpha_2),\ \beta=(\beta_1,\beta_2)$ and $\gamma=(\gamma_1,\gamma_2)$ contained in $U^2$ such that:

\begin{itemize}

\item $Y'(\alpha_1)=t a_1$, $Y'(\alpha_2)=t+a_2$, 

\item $Y'(\beta_1)=t b_1$, $Y'(\beta_2)=t+b_2$, and

\item $Y'(\gamma_1)=t a_1 b_1$, $Y'(\gamma_2)=t+a_1b_2+a_2$.
\end{itemize}
There is also a triple $(p,q,r)\in U^3$ such that: 

\begin{itemize}
\item $Y'(p)=t+b$,

\item $Y'(q)=t+a_1b+a_2$, and

\item $Y'(r)=t+a_1 b_1 b+b_1 a_2 + b_2$.
\end{itemize}
\end{lemma}

\begin{proof}
It just follows from the fact that $$\{t a_1, \ t+a_2 , \ tb_1,\  t+b_2,\  t a_1 a_2 ,\  t+a_1b_2+a_2,\  t+b, \ t+a_1 b+a_2,\  t+a_1 b_1 b + b_1 a_2 +b_2\}\subseteq B$$ and $$B\subseteq \{H_x(e,s):s\in U\}.$$
\end{proof}
For the rest of the section fix $B$, $\mathfrak a=(a_1,a_2,b_1,b_2,b)$ and $G_1=(\alpha,\beta,\gamma,p,q,r)$ as before.

\begin{prop}
$G_1$ is a field configuration for $\mathcal M$. 
\end{prop}

\begin{proof}

Using Fact \ref{derivativeDefinible} the $\mathcal G$-independence relations follows, so we only need to prove $\mathcal G$-dependence relations.

Let us prove for example that $q\in \text{acl}_{\mathcal G}(\alpha,p)$.

We have  
$$Y'(q)=t+a_1 b + a_2 = t + t^{-1} Y'(\alpha_1)(Y'(p)-t)+Y'(\alpha_2)-t.$$

Multiplying by $t$ we get:

\begin{equation}\label{pIsAlge}
t Y'(q)= Y'(\alpha_1)Y'(p)-t Y'(\alpha_1)+t Y'(\alpha_2).
\end{equation}

Let $d\in U$ such that $H_x(e,d)=t$. Call $X=X_{(d,h(d))}$ and define the family of curves

\begin{equation}\label{equationDefZ1}
Z_\delta:=(X_{\alpha_1}\circ X_{p})\ominus (X\circ X_{\alpha_1})\oplus (X\circ X_{\alpha_2})\ominus (X_{\delta}\circ X)
\end{equation}
 for $\delta\in Y$. 
 
 Here the $-$ is taken as in Definition \ref{AdditionAndCompositionDef}.
 %Note that $Z'_{q}((0,0))=0$ (because of Lemma \ref{additionAndCompositionLemma} and equation \ref{pIsAlg})
 
For $\delta\in Y$ define $$Z_\delta^{e}=\{x\in G:(x,e)\in Z_\delta\}$$ Since the family $(Z_\delta)_{\delta\in Y}$ is $\mathcal G$-definable with parameters $\alpha$ and $p$  it is enough to show that $$\{\delta\in Y:|Z_{\delta}^{e}|\leq|Z_{q}^{e}|\}$$ is finite. As we may assume that Morley degree of $Y$ computed in $\mathcal G$ is $1$, it is enough to show:

\begin{claim}\label{claimIntSube}
There is an open neighborhood $W$ containing $q$ such that for all $\delta \in W\setminus \{q\}$, $|Z_\delta^{e}|>|Z_q^{e}|$
\end{claim}
For this, list $Z_q^{e}=\{e=x_1,x_2,\ldots,x_n\}$ 
 
 Define the function $$F:U\times U_2\to K$$
 $$F(x,s)=H(H(x,p),\alpha_1) \oplus H(H(x,\alpha_1),d)^{-1} \oplus H(H(x,\alpha_2),d)\oplus H(H(x,d),s)^{-1}$$
 
 Where $U_2\subseteq U$ is some neighborhood of $q_1$ contained in $U$ such that everything is well defined.
 
 If $\delta\in U_2$ then the graph of the function $F(\_,\delta)$ is contained in $Z_\delta$ in a neighborhood of $0$.
 
 Using Equation (\ref{pIsAlge}) it is easy to see that  $F_x(e,q_1)=0.$ Therefore the function $F(\_,q_1)$ has a zero of multiplicity $d\geq 2$ at $x=e$. Apply Fact \ref{continuityOfRootsTheoremf} to the function $F$ and get open sets $V$ and $W$ with $(e,q_1)\in V\times W\subseteq U\times U_2$ such that for all $y\in W$ the function $F(\_,y)$ has exactly $d$ roots (counting multiplicities) at $V$.
 
Note that $F_x(e,s)=0$ just for finitely many $s\in U$. That is because 
$$F(x,s)=T(x)\oplus H(H(x,d),s)^{-1}$$
where $T(x)$ is a function depending just on $x$. Therefore $$F_x(e,s)=T'(e)+H_x(H(e,d),s)H_x(e,d)=T'(e)+H_x(e,s)t,$$
here we are using that $H(e,d)=e$. Thus if $F_x(e,s)$ is zero for infinitely many $s$ then $H_x(e,s)=-T'(e)/t$ for all such $s$ and then as the function $H_x(e,s)$ is $\mathbb K$ definable, Fact \ref{QEf} tells us that $\{s:H_x(e,s)=0\}$ contains an open ball and Fact \ref{identityTheoremf} implies that $s\mapsto H_x(e,s)$ is a constant function  on $U$ but this contradicts clause 5 of the definition of a good family of curves.

Therefore we may assume that for each $s\in W\setminus\{q_1\}$ the function $F(\_,s)$ has a simple zero at $x=e$.

Because of our choice of $V$ and $W$, for each $s\in W$ the function $F(\_,s)$ has at least two zeros at $V$ (counting multiplicities) and if $s\neq q_1$, as $F(\_,s)$ has only simple roots, there are two different points $c_1,c_2\in V$ such that $F(c_i,s)=0$ for $i=1,2$. Thus, for any $\delta\in W\setminus\{q_1\}$ there are at least two different $x\in V$ such that $x\in Z^{0}_{\bar\delta}$. 

\begin{claim}\label{claimNoPierdoCeros}
For $i>1$ take pairwise disjoint open sets $V_i$ containing $x_i$. Remember that $x_i$ are the elements of $Z_q^{e}$. Then we can find open sets $W_i\subseteq U$ with $(x_i,q_1)\in V_i\times W_i$ such that for each $s\in W_i$ there is at least one $x\in Z_{s}^0\cap V_i$.
\end{claim}

\begin{proof}(Proof of Claim \ref{claimNoPierdoCeros})
Call $$P:=(X_{\alpha_1}\circ X_{p})\ominus (X\circ X_{\alpha_1})\oplus (X\circ X_{\alpha_2}).$$ 
So there is some $y_i$ such that $$(x_i,y_i)\in P \cap X_q\circ X.$$ Then there is some $z_i$ such that $(x_i,z_i)\in X$ and $(z_i,y_i)\in X_q$.

We can find am analytic function $\Phi_1(x,s)$ and a natural number $n_1$ such that $$\left(x,\text{Fr}^{-n_1}(\Phi_1(x,s))\right)\in X_s$$ for all $x$ in some neighborhood of $x_i$ and all $s$ in a neighborhood of $e$. In the same way we can find $\Phi_2$ and $n_2$ such that $$\left(z,\text{Fr}^{-n_2}(\Phi_2(z,s))\right)\in X_s$$ for all $z$ in a neighborhood of $z_i$ and all $s$ in a neighborhood of $q$.

For $j=1,2$ let $\Psi_j=\text{Fr}^{-n_j}\circ \Phi_j$.

Therefore $$(x,\Psi_2(\Psi_1(x,d),q))\in X_q \circ X$$ for all $x$ in some neighborhood of $x_i$.  Moreover if $\delta$ is close enough to $q_1$ then the graph of $$x\mapsto \Psi_2(\Psi_1(x,d),\delta)$$ is contained in $X_\delta \circ X$.

Call $$\Psi(x,s)=\Psi_2(\Psi_1(x,d),s)$$.

As $(x_i,y_i)\in P$ and $P$ has no isolated points there is a polynomial $G(x,y)$ such that for $(x,y)$ close enough to $(x_i,y_i)$ if $G(x,y)=0$ then $(x,y)\in P$.

So we would be done if we are able to prove that for all $\delta$ close enough to $q_1$ there is some $x\in V_i$ such that $$G(x,\Psi(x,\delta))=0.$$ 
For this note that $$\Psi(x,z)=\text{Fr}^{-n}(\tilde \Phi_2(\Phi_1(x,d),z)$$
where $n=n_1+n_1$ and $\tilde \Phi_2$ is a function such that $\text{Fr}^{-n_1}(\Phi_2(x,s))=\tilde\Phi_2(\text{Fr}^{-n_1}(x),s)$ for all $(x,s)$ so $\tilde\Phi_2$ is the function obtained by changing all the coefficients of $\phi_2$ for its $p^{n_1}$-power. That is, if $$\Phi_2(x,s)=\sum_I a_I x^{i_1} s ^{i_2},$$ then $$\tilde \Phi_2(x,s)=\sum_I a_I^{p^{n_1}} x^{i_1} s^{i_1 p^{n_1}}.$$

Let $$\tilde \Psi(x,s):=\tilde \Phi_2(\Phi_1(x,e),z)$$
an analytic function.

So we want to prove that for $\delta$ close enough to $q_1$ the function $G(x,\text{Fr}^{-n}(\tilde \Psi(x,\delta)))$ has a zero at $V_i$ but again $$G\left(x,\text{Fr}^{-n}(\tilde \Psi(x,z))\right)=\text{Fr}^{-n}\left(\tilde G(x,\tilde\Psi(x,z))\right)$$ where $\tilde G$ is obtained obtained from $G$ changing all the coefficients for it's $p^n$-power. That is, if $$G(x,y)=\sum_I a_{I} x^{i_1} y^{i_2},$$ then $$\tilde G(x,y)=\sum_I a_{I}^{p^n} x^{i_1 p^n} y^{i_2.}$$

Now as $\tilde G(x,\tilde \Psi(x,z))$ is an analytic function that is zero at $(x_i,q_1)$ we can apply Fact \ref{continuityOfRootsTheoremf} and find some neighborhood $W_i$ of $q_1$ such that for all $\delta\in W_i$ there is $x\in V_i$ with $\tilde G(x,\Psi(x,\delta))=0$ but then $$G\left(x,\text{Fr}^{-n}(\tilde \Psi(x,\delta))\right)=\text{Fr}^{-n}\left(\tilde G(x,\tilde \Psi(x,\delta))\right)=\text{Fr}^{-n}(0)=0.$$ 

\textit{End of proof of Claim \ref{claimNoPierdoCeros}}
\end{proof}

Then if $\delta\in \bigcap W_i\setminus \{q_1\}$, $Z_{\delta}^0$ contains at least two zeros at $V$ and at least one at $V_i$ for $i>1$ so it has at least $n+1$ points and since there are infinitely many such $z$'s, we are done. 

\end{proof}

\chapter{Additive Case}\label{additiveCase}

In this chapter we prove:

\begin{thm}\label{propAdditive}

Let $(G,+)$ be an infinite $\mathbb K$-definable subroup of $(K,+)$ group and let $X\subseteq G^2$ be a  non affine $\mathbb K$-definable set of dimension one. Assume that the structure $\mathcal G=(G,\oplus,X)$ is strongly minimal, then $\mathcal G$ interprets an algebraically closed field that is definable isomorphic (in $\mathbb K$) to $(K,+,\cdot,0,1)$. 
\end{thm}

The fact that the field is isomorphic to $K$ follows from Theorem 7.1 of \cite{HYField}. So we will devote the entire chapter to find an interpretable field.

Throughout this chapter we fix $G$ and $X\subseteq G^2$ satisfying the assumptions of Theorem \ref{propAdditive}. Without loss of generality we may assume that $(0,0)\in X$.

Note that in this case Assumption \ref{assmOperationSumDerivative} just says that $(f+g)'(0)=f'(0)+g(0)$ and $(-f)'(0)=-(f'(0))$ for each pair of functions $f$ and $g$ analytic at $0$. Which is clearly true.

Therefore it is enough to provide a good family of curves definable in $\mathcal G$.

 %1- Descomposicion de X, es localmente una funcion analitica (necesito dF/dY=0 solo en finitos puntos pero espero que baste con elegir bien que coordenada privilegiar, como en HS)
 %2- La derivada de X tiene imagen infinita.
 %2- Compatibilidad entre composicion/suma de funciones y sus derivadas
 %3- Podemos asumir dX/dx y $dZ_b/dx$ son cero solo en finitos puntos de X y de $Z_b$ respectivamente

Toward that end we need some preliminaries on power series.

\section{Power Series}

\begin{defi}\label{defN}
Let $$f(x)=\sum_{n\geq 1} b_n x^n$$ be a power series converging in a neighborhood $D(f)$ of $0$. For $a\in D(f)$ define $f_a(x)=f(x+a)-f(a)$. Suppose $$f_a(x)=\sum_{n\geq 1}b_{n,a} x^n$$ and define 
\begin{equation*}
\begin{split}
s_n(f) &:=\{b_{n,a}:a\in D(f)\}  \\ 
N(f) &:=\min\{n:s_n(f)\text{ is infinite}\}
\end{split}
\end{equation*}
\end{defi}

\begin{lemma}\label{N}
Let $f(x)$ be a function analytic at $0$ such that $f(0)=0$ and $f'(0)\neq 0$. Let $g(x)$ be an analytic inverse for $f$ converging in some neighborhood of $0$. Then if $N(f)$ is finite, $N(g)\geq N(f)$.
\end{lemma}

\begin{proof}
The existence of $g$ it is just Fact \ref{inversionf}. 

The inequality $N(g)\geq N(f)$ is a consequence of the usual Lagrange inversion formula but we present a proof anyway. We will prove that $c_n$ depends only on $b_1,b_2,\ldots,b_n$ by induction on $n$. For $n=1$ it is easy to see that $c_1=1/b_1$. Assume therefore that $c_n$ depends only on $b_1,\ldots, b_n$ and we will prove it for $n+1$. As we have that $$x=f(g(x))=\sum_i b_i\left(\sum_j c_j x^j\right)^i$$ the coefficient of $x^{n+1}$ in the right side of equality has to be $0$. But this coefficient is 
$$b_1 c_{n+1} + b_2 d_2 +\ldots+b_n d_n+b_{n+1}c_1^{n+1},$$ 
where $d_2,\ldots,d_n$ are some polynomials quantities depending only on $c_1,\ldots,c_n$ that on it's turn, by induction hypothesis, depend only on $b_1,\ldots,b_n$. Therefore $$c_{n+1}=\frac{-b_2 d_2 - \ldots -b_n d_n - b_{n+1} c_1^n}{b_1}$$ depends only on $b_1,\ldots,b_n,b_{n+1}$.

Thus $c_{n,a}$ depends only on $b_{1,a},\ldots,b_{n,a}$ and if $k\leq N$ $b_{k,a}$ is constant as $a$ varies. Therefore if $n\leq N$ $c_{n,a}$, is constant as $a$ varies and then $N(g)\geq N$.

\end{proof}

\begin{corollary}
Assuming notation of Lemma \ref{N}. If $N(f)$ and $N(g)$ are both finite then $N(f)=N(g)$.
\end{corollary}
\begin{proof}
By Lemma \ref{N} one has that $N(f)\leq N(g)$ apply same lemma for $g$ and get that $N(g)\leq N(h)$ where $h$ is the inverse of $g$ but then $h=f$ so $N(f)=N(g)$.
\end{proof}

\begin{lemma}\label{lemmaN}
Assume $$f=\sum_{n\geq 1}b_n x^n$$ is an analytic function converging in a neighborhood of $0$. Assume moreover that $N(f)$ is finite. Let $$l=\min \{e\in \mathbb N: \exists n>1 \text{ such that }p\nmid n\wedge  b_{np^e}\neq 0\}.$$
Then for all $k<l$ and all $m\geq 1$ such that $p\nmid m$ the coefficient $b_{mp^k,a}$ is constant as $a$ varies. Moreover $N(f)=p^l$. 
\end{lemma}

\begin{proof}
Assume notation of Definition \ref{defN} and let $l$ be as in the statement of the lemma. It is well defined because of our assumption on $N(f)$: If $$\{e\in \mathbb N: \exists n>1 \text{ such that }p\nmid n\wedge  b_{np^e}\neq 0\}=\emptyset$$ then $f$ is a function in which the degree of all the non zero monomials are powers of $p$ so $f(x+a)=f(x)+f(a)$ and therefore $s_n(f)$ is finite for all $n$.
 
We will show that for all $k<l$ and all $m\geq 1$ with $p\nmid m$ the coefficient $b_{mp^k,a}$ is constant as $a$ varies. And we will also show that the coefficient $b_{p^l,a}$ assumes infinitely many values as $a$ varies. 
 
Let $n>1$ such that $b_{np^l}\neq 0$ and $p\nmid n$. Therefore $$f(x)=b_1 x + b_{p^2} x^{p^2} + \ldots + b_{p^l} x^{p^l} + \sum_{i>p^l} b_i x^i.$$ 

We will show first that if $k<l$ and $m\geq 1$ with $p\nmid m$, then the coefficient $b_{mp^k,a}$ is constant as $a$ varies. Remember that $b_{mp^k,a}$ is the coefficient of $x^{mp^k}$ in the expansion of $f_a(x)=f(x+a) - f(a)$. In this expansion $x^{mp^k}$ just appears in the terms of the form  $b_t (x+a)^t$ with $t\geq mp^k$. Write $t=u p^s$ with $p\nmid u$. Suppose first $u>1$ then if $s<l$, $b_t=0$. If $s\geq l>k$ then $(x+a)^t=(x^{p^s} + a^{p^s})^u$. And therefore in the expansion of $(x+a)^t$ all the exponents of $x$ are multiples of $p^s$ so non of those equals $mp^k$ (as $s>k$ and $p\nmid m$). We can assume then that $u=1$ but in this case $(x+a)^t=x^{p^s} +a^{p^s}$ and the only way the term $x^{mp^k}$ appears is $k=s$ and $m=1$. In this case the coefficient of $x^{mp^k}$ in $f_a(x)$ is $b_{mp^k}$ so it is constant as $a$ varies.\\

We show now that the coefficient $b_{p^l,a}$ is not constant as $a$ varies. As $$f(x)=b_1 x + b_{p^2} x^{p^2} + \ldots + b_{p^l} x^{p^l} + \sum_{i>p^l} b_i x^i$$ and $b_{np^k}=0$ for all $n>0$ and $k<l$ the function $\displaystyle\sum_{i>p^l} b_i x^i$ is indeed a function in the variable $x^{p^l}$. So $$f(x)=b_1 x + \ldots + b_{p^l} x^{p^l} +\sum_{i>1} b_{i p^l}x^{ip^l}$$ therefore the coefficient of $x^{p^l}$ in $f_a$ is $$b_{p^l,a}=b_{p^l} + \sum_{i>1} b_{i p^l} (a^{p^l})^i $$ but this is a series in the variable $a$ with some non zero coefficients, using Fact \ref{identityTheoremf} we conclude that it is not constant as $a$ varies and we conclude. 
\end{proof}

\section{Getting a good family of curves on $\mathcal G$}

In this section we prove:

\begin{prop}\label{assmFamily}
There is a good family of curves (as in Definition \ref{defGoodFamily}) definable in $\mathcal G$.
%\begin{enumerate}
 %   \item A $\mathcal G$-definable set $Y\subseteq G\times G$ and a $\mathcal G$-definable family of curves $(X_a)_{a\in Y}$
  %  \item $U$, an open neighborhood of $0$, and an analytic function $H(x,s)$ defined in $U\times U$ 
   % \item An analytic function $h:U\to K$
%\end{enumerate}
%Such that
%\begin{enumerate}
    
 %   \item For all $s\in U$ if we define $\bar s:=(s,h(s))$ then $\bar s\in Y$
    
    %\item For all $(x,y)\in Y$ there is an analytic function $g$ defined in a neighborhood $U_x$ of $x$ and a natural number $n$ such that for all $z\in U_x$ if we define $\bar z:=(z,\text{Fr}^{-n}(g(z)))$ then $\bar z\in Y$
    
  %  \item Each $X_a$ is a one dimensional subset of $K^2$ containing $(0,0)$ with no isolated points.
    %\item For all $s\in U$ and for all $x\in U'$, $(x,H(x,s))\in X_{\bar s}$
    
   % \item $H(0,s)=0$ for all $s\in U$
    %\item The set of derivatives $\displaystyle\{H_x (0,s):s\in U\}$ is infinite
    
%    \item For each $s\in U$ and $(x,y)\in X_{\bar s}$ there is a neighborhood $V_x$ of $x$ and a neighborhood $U_s\subseteq U$ of $s$, an analytic function $\Phi(z,\delta)$ defined in $V_x\times U_s$ and a natural number $n$ such that for all $(w,\delta)\in V_x\times U_s$ one has that $(w,\text{Fr}^{-n}(\Phi(w,\delta)))\in X_{\bar \delta}$.
%\end{enumerate}
\end{prop}

\begin{proof}
Take $Y$, $U$ and $h$ as provided by Proposition \ref{assmGraphpC} applied to the structure $\mathcal K$.
If $s_1(h)=\{h'(z):z\in U\}$ is infinite, then for $a\in Y$ define $X_a:=Y_a$ (as in Definition \ref{AdditionAndCompositionDef}). So we can take $$H(x,s)=h(x-s)+h(s)$$ and if $s\in U$ and $(x,y)\in Y_{\bar s}$ by the conclusion of Proposition \ref{assmGraphpC} there is an analytic function $g$ defined in a neighborhood of $x$ and a natural number $n$ such that $(w,\text{Fr}^{-n}(g(w)))\in Y$ for all $w$ in that neighborhood. Now using Lemma \ref{additionAndCompositionLemma}, if $\delta$ is close enough to $0$ then  $$(w-\delta,\text{Fr}^{-n}(g(w))-h(\delta))\in Y_{\bar \delta}$$ so we can take $$\Phi(w,\delta)=g(w+\delta)-\text{Fr}^{n}(h(\delta)).$$

Now assume that $s_1(h)$ is finite. Therefore there is an open neighborhood $U$ containing $0$ such that $h'(w)$ is constant in $U$. Using Fact \ref{identityTheoremf} we have that $h(w)= Kw + f(w)$ where $K$ is some constant and $f(w)$ is a function in the variable $w^{p^n}$ for some $n>0$. Take $n$ maximal with that property. So $h(w)=Kw + F(w^{p^n})$ for some analytic function $F$.
For each $a\in Y$ define 
\begin{equation}\label{equationDefGood}
X_a:= (Y-Y_a)\circ (Y-Y_c)^{-1}
\end{equation}
where $c=(c_1,h(c_1))$ for some fixed $c_1\in U$.

\begin{claim}\label{xIsGoodFamily}
The family $(X_a)_{a\in Y}$ defined by Equation \ref{equationDefGood} is a good family of curves for $\mathcal G$.
\end{claim}

Suppose that $$h(x)= Kx + \sum_{i\geq 1}b_{ip^n}x^{ip^n}$$
with $b_{p^n}\neq 0$. Moreover let $m$ such that $N(f)=p^m$ (using Lemma \ref{lemmaN}). So there is an open neighborhood $U_0$ of $0$ such that for all $a\in U_0$, if we define $$h_a(x)=h(x+a)-h(a)$$ and put $$h_a(x)=\sum_{n\geq 1} b_{n,a}x^n$$ then $$b_{n,a}=b_n$$ for all $n$ such that $s_n(h)$ is finite. Therefore $Y-Y_a$ contains the graph of $$h-h_a=\sum_{i\geq 1}d_{i,a}x^{ip^m}$$
where $d_{i,a}=b_{ip^m}-b_{ip^m,a}$. Note that $\{d_{1,a}:a\in U\}$ is infinite. In the same fashion for all $b_1\in U$ if we call $b=(b_1,h(b_1))$, then $Y-Y_b$ contains the graph of $$h-h_b=\sum_{i\geq 1}d_{i,b}x^{ip^m}$$ we can pick $b_1$ such that $d_{1,b}\neq 0$. 

Now define $$G_a(x)=\sum_{i\geq 1}d_{i,a} x^i$$ so $$(h-h_a)(x)=G_a(x^{p^m}).$$ As $d_{1,b}\neq 0$ then $G'_b(0)\neq 0$ so using Fact \ref{inversionf}, exists an analytic function $G^{-1}_b$ defined in some neighborhood of $0$ such that $G_b(G_b^{-1}(z))=z$ for all $z$ in that neighborhood.

\begin{claim}\label{claimGraphIsContained}
The graph of $G_a \circ G_b^{-1}$ is contained in $(Y-Y_{\bar a})\circ (Y-Y_b)^{-1}=X_{\bar a}$.
\end{claim}
\begin{proof}(Proof of Claim \ref{claimGraphIsContained})

Let $y=G_a(G_b^{-1}(x))$ and we will prove that $(x,y)\in (Y-Y_{\bar a})\circ (Y-Y_b)^{-1}=X_{\bar a}$. So we have to show that there is some $z$ with $(z,x)\in Y-Y_b$ and $(z,y)\in Y-Y_{\bar a}$. Let $z$ such that $z^{p^m}=G_b^{-1}(x)$. Then, $$(h-h_b)(z)=G_b(z^{p^m})=G_b(G_b^{-1}(x))=x$$ and as the graph of $h-h_b$ is contained in $Y-Y_b$ then $(z,x)\in Y-Y_b$. In a similar way $$(h-h_a)(z)=G_a(z^{p^m})=G_a(G_b^{-1}(x))=y$$ and as the graph of $h-h_a$ is contained in $Y-Y_{\bar a}$ then $(z,y)\in Y-Y_{\bar a}$ so $(x,y)\in (Y-Y_{\bar a})\circ (Y-Y_b)^{-1}$

\textit{(End of proof of Claim)}
\end{proof}

Now note that  

$$G_b^{-1}(x)=\frac{1}{d_{1,b}}x + L(x)$$ for some analytic function $L$ such that every non zero monomial of $L$ has degree bigger than $1$.

Therefore $$G_a\circ G_b^{-1} (x) = \frac{d_{1,a}}{d_{1,b}} x + \sum_{i>1} e_{i,a} x^{i}$$ For some coefficients $e_{i,a}$. It is easy to see that each $e_{i,a}$ is a power series in the variable $a$, and the same is true for $d_{1,a}$. Therefore there is a power series $H(x,a)$ such that for all $a$ in a neighborhood of $0$ one has that $H(x,a)=(G_a\circ G_b^{-1})(x)$ so $$H_x(0,a)=(G_a\circ G_b^{-1})'(0)=\frac{d_{1,a}}{d_{1,b}}$$ that takes infinitely many values as $a$ varies. 

So we have already proved clauses 1-5 in the definition of good family of curves and only 6 is missing. For this let $s\in U$ and let $(x,y)\in (Y-Y_{\bar s})\circ (Y-Y_{b})^{-1}$ so there is some $z$ such that $(z,x)\in Y-Y_b$ and $(z,y)\in Y-Y_{\bar s}$. There are $y_1,y_2$ such that $(z,y_1)\in Y$, $(z,y_2)\in Y_{\bar s}$ and $y=y_1-y_2$. Again Proposition \ref{assmGraphpC} there are analytic functions $g_1$, $g_2$ defined in some neighborhoods $U_1$ and $U_2$ of $0$, and natural numbers $n_1$, $n_2$

$\text{Fr}^{-n_1}(g_1(z))=y_1$, 
$\text{Fr}^{-n_2}(g_2(z))=y_2$,

and the graph of $\text{Fr}^{-n_1}\circ g_1$ and $\text{Fr}^{-n_2}\circ g_2$ are contained in $Y$ and $Y_{\bar s}$ respectively.

Moreover for $\delta$ close enough to $s$, the graph of  $w\mapsto \text{Fr}^{-n_2}(g_2(w-s+\delta))+h(s)-h(\delta)$ for $w\in U_2$ is contained in $Y_{\bar\delta}$. For showing that we have to prove that if $$y=\text{Fr}^{-n_2}(g_2(w-s+\delta))+h(s)-h(\delta)$$ then $(w,y)\in Y_{\bar\delta}$  but this is equivalent to show that $(w+\delta,y+h(\delta))\in Y$ and again this is the same to show that $(w+\delta-s,y+h(\delta)-h(s))\in Y_{\bar s}$. And as the graph of $\text{Fr}^{-n_2}\circ g_2$ is contained in $Y_{\bar s}$ it is enough to prove that  
$$\text{Fr}^{-n_2}(g(w+\delta-s))=y+h(\delta)-h(s)$$ but this is precisely the definition of $y$.

Note that $\text{Fr}^{-n_2}(g_2(w-s+\delta))+h(s)-h(\delta)=\text{Fr}^{-n_2}(G(w,\delta))$ where $$G(w,\delta)=g_2(w-s+\delta)+h(s)^{p^{n_2}}-h(\delta)^{p^{n_2}},$$ an analytic function.

For all $(z,x)\in Y-Y_b$ there is an analytic function $g$ converging in some neighborhood of $z$ an a natural number $k$ such that 
$\text{Fr}^{-k}(g(z))=x$ and the graph of $\text{Fr}^{-k}\circ g$ is contained in $Y-Y_b$ in some neighborhood of $z$. Moreover we can assume that $g'$ is never zero. So there exists an inverse for $g$, say $f$, converging in a neighborhood of $x^{p^k}$ therefore  the graph of $f\circ \text{Fr}^{k}$ is contained in $(Y-Y_b)^{-1}$ so, for $\delta$ close enough to $s$, the graph of $$w\mapsto \text{Fr}^{-n_1}(g_1(\hat w)) -\text{Fr}^{-n_2}(G(\hat w,\delta))$$ where $\hat w=f(\text{Fr}^{k}(w))$ is contained in $(Y-Y_{\bar\delta})\circ (Y-Y_b)^{-1}$.

Assume $n_1\geq n_2$  so $$\text{Fr}^{-n_1}(g_1(\hat w)) -\text{Fr}^{-n_2}(G(\hat w,\delta))=\text{Fr}^{-n_1}\left( g_1(\hat w)- \text{Fr}^{n_1-n_2}(G(\hat w,\delta))\right)$$
We can take $n=n_1$ and $$\Phi(w,\delta)=g_1(\hat w)- \text{Fr}^{n_1-n_2}(G(\hat w,\delta))$$ in order to get clause 6 of the definition.
\end{proof}

%falta probar que la función z\mapsto Y'(z) es KK-definible

%Finally for proving Theorem \ref{propAdditive} we need the following strong theorem on valued fields:

%\begin{thm}\label{continuityOfRootsTheorem}(Continuity of Roots, Theorem 2.4 in \cite{KR})

%Assume $K$ is complete. Suppose $U_1, U_2$ are open subsets of $K$ and $$F:U_1\times U_2\to K$$ is an analytic  function. Assume that for some $(a,b)\in U_1\times U_2$ it is true that the function $F(\_,b)$ has a zero of multiplicity $d>0$ at $a$. Then there are open sets $U_a$ and $U_b$ with 
%$a\in U_a\subseteq U_1$ and
%$b\in U_b\subseteq U_2$ such that:

%\begin{enumerate}
%    \item $a$ is the only $x\in U_a$ such that $F(x,b)=0$
 %   \item For each $y\in U_b$ the function $F(\_,b)$ has exactly $d$ zeros in $U_a$ (counting multiplicities)
%\end{enumerate}

%\end{thm}

By Theorem \ref{thmFieldGivenFamily} and Proposition \ref{assmFamily} we complete the proof of Theorem \ref{propAdditive}

\chapter{Multiplicative Case}\label{multiplicativeCase}

In this chapter we call $M=K\setminus \{0\}$ the universe of the multiplicative group of $K$ and $(\cdot)$ denotes the multiplication on $M$.

\begin{defi}
Let $(H,\otimes)$ be a one dimensional group interpretable in $\mathbb K$. We say that $H$ is locally isomorphic to $(M,\cdot)$ if there is $U\subseteq M$, some open set containing $1$, and $i:U\to H$, a $\mathbb K$-definable injection such that $i(1)$ is the neutral element of $H$ and for all $x,y\in U$ if $x\cdot y\in U$ then $i(x\cdot y)=i(x)\otimes i(y)$.
\end{defi}

The main result of the chapter is:

\begin{thm}\label{thmMultiplicative}
Let $(H,\otimes)$ be a $\mathbb K$-interpretable group of dimension one. Suppose that $H$ is locally isomorphic to $(M,\cdot)$. Assume moreover that $X\subseteq H^2$ is a $\mathbb K$-inerpretable set that is not a boolean combination of subgroups of $(H,\otimes)\times (H,\otimes)$, and assume $\mathcal H=(H,\otimes,X)$ is strongly minimal, then $\mathcal M$ interprets a field, moreover such a field is definable isomorphic (on $\mathbb K$) to $K$.
\end{thm}

Again the fact that the field is isomorphic to $K$ follows from Theorem 7.1 of \cite{HYField} so we just have to find an interpretable field.

Throughout this chapter fix $H$ and $X\subseteq H^2$ as in the hypothesis of Theorem \ref{thmMultiplicative} and call $\mathcal H := (H,\otimes,X)$.  We also fix $U\subseteq M$ an open set containing $1$ and $i:U\to H$ given by the definition of locally isomorphic to $(M,\cdot)$

Let $1_H$ be the neutral element of $H$ so without loss of generality we shall assume that  $(1_H,1_H)\in X$.

For $Y\subseteq H^n$ let $Y^*\subseteq B^n$ be the inverse image of $Y$ via $i$, that is:
$$Y^{*}=\{(y_1,\ldots,y_n)\in B^n :(i(y_1),\ldots,i(y_n))\in Y\}$$ and for $y\in H^n$ let $y^{*}\in B^n$ such that $i(y^*)=y$.

Therefore $X^*$ is a $\mathbb K$-definable subset of $K$ of dimension one 

By using same definition as in proof of Proposition \ref{assmGraphpC} there is a $\mathcal H$-definable set $Y\subseteq H^2$, an open set $V$ contained in $U$, $V\ni 1$ and analytic function $h:V\to M$ such that $h(1)=1$ and $(x,h(x))\in Y^*$ for all $x\in V$.

%Moreover for each $(x,y)\in Y^*$ there is a natural number $n$ and an analytic function $g$ defined in a neighborhood $U_x$ of $x$ such that $y=\text{Fr}^{-n}(g(x))$ and $(z,\text{Fr}^{-n}(g(z)))\in Y^*$ for all $z\in U_x$.

Changing $U$ by $U\cap V$ (and $i$ by $i\vert_{U\cap V}$) we may assume that $U=V$.

Assume for a moment that $s_1(h):=\{h_a'(1):a\in U\}$ is infinite. In this case we can get a good family of curves definable in $\mathcal M$ just as in the additive case taking $H(x,s)=h(xs)/h(s)$. Therefore we can use the same definition as in Theorem \ref{thmFieldGivenFamily} and get a field interpretable in $\mathcal H$. So we may assume that $s_1(X)$ is finite and by shrinking $B$ we may assume:

\begin{assm}\label{assmDerivativeConstant}
 $h'_a(1)$ is constant as $a$ varies in $B$.
 \end{assm}
 
 Under Assumption \ref{assmDerivativeConstant} we will find a group $G$ interpretable in $\mathcal H$ that is (locally) isomorphic to $(K,+)$ therefore we use the results of Chapter \ref{additiveCase} to find a field interpretable in $G$ and therefore in $\mathcal H$.

 In the first section of this chapter we find such a group. In the second section we use that group for finding a field interpretable in $\mathcal H$.
 
 \section{Finding a Group}\label{findingAGroup}
 \begin{prop}\label{propGroupInterpretableAdditive}
 Under Assumption \ref{assmDerivativeConstant}, in any $V$, open neighborhood of $1$, there is a group tuple of elements of $U$ such that it's image via $i$ is a configuration $\mathcal H$ that is inter algebraic some group configuration of $(K,+)$.
 \end{prop}

 \begin{proof}
Remember that we fixed $Y\subseteq H^2$ a one dimensional $\mathcal H$-definable subset, and $h:U\to M$ an analytic function such that for all $x\in U$, $(x,h(x))\in Y^*$. Replacing $B$ by $B\cap U$ we assume that $B=U$ and replacing $V$ for $V\cap B$ may assume that $V\subseteq B$.

Note that if $a\in U$ then the graph of $h_a(x):=h(xa)/h(a)$ is contained in $(Y^{*})_{(a,h(a))}$ for $x$ in some neighborhood of $1$. For $a\in U$ we will write $Y^{*}_a$ instead of $(Y^*)_{(a,h(a))}$.

As $h_a$ is analytic there is some power series converging in a neighborhood of $1$ that represents $h_a$, say 
\begin{equation}\label{haDefEquation}
h_a(x)=1+\sum_{n\geq 1} d_n(a) (x-1)^n.
\end{equation}

By Assumption \ref{assmDerivativeConstant} we know that $d_1(a)$ is finite as $a$ varies in $U$. So we may assume it is constant. Let $N$ be the minimum natural number such that 
$\{d_N(a):a\in U\}$ is infinite. 

We may assume that for $n<N$ the set $\{d_n(a):a\in U\}$ is a singleton.

Exchanging $Y$ by $Y_c\circ Y_b^{-1}$ for $(b,c)$ generic enough we may assume that for  all $a\in U$, $d_1(1)=d_1(a)=1$ and  for $1<n<N$, $d_n(1)=d_n(a)=0$. That follows from a straightforward computation of the coefficients of $h_a\circ h_b^{-1}$.

Note that if $a$ and $b$ are close enough of $1$ then the function $h_a\circ h_b$ is analytic in some neighborhood of $1$. So if the power expansion of $h_a\circ h_b$ is given by:
\begin{equation}\label{equationComposition}
    h_a(h_b(x))=1 +\sum_{n\geq 1} d_n(a,b)(x-1)^n,
\end{equation} 
Then we have:
\begin{claim}\label{claimSumaDerivadas}
 $d_N(a,b)=d_N(a)+d_N(b)$.
\end{claim}

\begin{proof}(Proof of Claim \ref{claimSumaDerivadas}) 
Note that

$$
\begin{array}{ll}
  h_a(h_b(x))& =1+ \sum_{n\geq 1} d_n(a)\left (1+\sum_k b_k(b) (x-1)^k - 1\right)^n \\
     & =1+ \sum_{n\geq 1} d_n(a)\left (\sum_k b_k(b) (x-1)^k \right)^n
\end{array}
$$

So the coefficient that multiplies $(x-1)^N$ is 
$$d_N(a,b)=d_1(a) b_N(b) + d_2(a)p_1(b) + \cdots + d_N(a)(b_1(b))^N$$

Where for $1<i<N$ $p_i$ is some polynomial on $b$. Given that for each such $i$ $d_i(a)=0$ and $d_1(a)=d_1(b)=1$ then $$d_N(a,b)=d_N(a)+d_N(b)$$ as desired.
\end{proof}
Note that for each $n\in \mathbb N$ the map from $U$ to $K$ given by $a\mapsto d_n(a)$ is $\mathbb K$-definable. So, as $\{d_N(a):a\in U\}$ is an infinite $\mathbb K$-definable subset of $K$, it contains an open ball, say $B$ with center $t$. Assume $t=d_N(1)$ and call $B_0=\{x-t:x\in B\}$ a ball around $0$. 

Find $a,b,e\in U$ generic independent (in the sense of $\mathbb K$). Then $d_N(a)-t$ and $d_N(b)-t$ are in $B_0$, their sum is in $B_0$ so $t+(d_N(a)-t)+(d_N(b)-t)=d_N(a)+d_N(b)-t\in B$ so there is some $c\in U$ such that $d_N(c)=d_N(a)+d_N(b)-t$. Fix such $c$. 

In the same way  there is $f\in U$ such that $d_N(f)=d_N(a)+d_N(e)-t$. And there is also $g$ such that $d_N(g)=t+(d_N(a)-t)+(d_N(b)-t)+(d_N(c)-t)$.

Let $G_2=(d_N(a)-t,d_N(b)-t,d_N(c)-t,d_N(e)-t,d_N(f)-t,d_N(g)-t)$. It is a group configuration for $(K,+)$ in $\mathbb K$.
Define $G_1=(i(a),i(b),i(c),i(e),i(f),i(g))$. Note that $G_1$ is interdefinable with $G_2$ (in $\mathbb K$) just because the map $z\mapsto d_N(i^{-1}(z))-t$ is $\mathbb K$ definable. 

\begin{claim}\label{claimGC}
 $G_1=(i(a),i(b),i(c),i(e),i(f),i(g))$ is a group configuration for $\mathcal H$.
\end{claim}
\begin{proof}(Proof of Claim \ref{claimGC})

Let us see for example that $i(c)\in \acl_\mathcal H(i(a),i(b))$. Let $Z=Y_{i(a)}\circ Y_{i(b)}\circ Y^{-1}$. This is a $\mathcal H$-definable set with parameters $(i(a),i(b))$.

We will prove that if we define 
$$E_{Z}:=\{e\in Y:|Y_e\cap Z|>|Y_{i(c)}\cap Z|\},$$ 
then it is an infinite set. 

For this we will prove first that  $$F_{Z}:=\{e\in Y^{*}:|Y^{*}_e\cap Z^*|>|Y^*_{c}\cap Z^*|\}$$ is an infinite set.

% As it is $\mathcal H$-definable and strongly minimal, it is cofinite and as $c$ is on it's complement one gets the result.

For proving that list 
$$Y^*_{c}\cap Z^*=\{x_1=(1,1),x_2,\ldots,x_n\},$$
and for $i=1,\ldots,n$ take $W_i=W'_i\times W''_i\subseteq U\times U$ pairwise disjoint open sets with $x^{*}_i\in W^*_i$. We will prove that there is an open subset $V\subseteq U$ containing $c$ such that for all $e\in V\setminus\{c\}$ the curve $Y^*_e$ intersects $Z^*$ at two distinct points in $W_1$ and at least one inside of $W_i$ for $i>1.$ 

First note that there are just finitely many $z\in U$ such that $d_n(z)=d_N(a)+d_N(b)-t$. Otherwise there will be an open set where the equality holds an then by Fact \ref{identityTheoremf} $d_N(z)$ is constant as $z$ varies in $U$ and we are assuming that this is not the case. So we may assume that $d_N(z)\neq d_N(a)+d_N(b)-t$ for $z\in U\setminus\{c\}$

Now define the function $H(x,e)=h_a\circ h_b\circ h^{-1}- h^{-1}_e(x)$ that is analytic in some neighborhood of $(1,c)$. As the function $H(\_,c)$ has a zero of multiplicity $N$ at $x=1$. Using Fact \ref{continuityOfRootsTheoremf} we find some open sets $V_1\subseteq W'_1$ and $U_1\subseteq U$ with $(1,c)\in  V_1\times U_1$ such that for each $e\in U_1$ the function $x\mapsto H(x,e)$ has $N$ zeros (counting multiplicities) in $V_1$. Note that the function $x\mapsto H(x,e)$ has a zero of order $N$ at $x=1$ just for finitely many $e$, otherwise by \ref{identityTheoremf} $d_N(a)$ should be constant in some open and we are assuming that it is not the case. So, by shrinking $U$ we may assume that if $e\neq c$ there is no zero of multiplicity $N$ at $U$. Therefore there are at least two different zeros at $V_1$. That is, there is two different points $x,x'\in W'_1$ such that $y:=h_a\circ h_b\circ h^{-1}(x)=h_c(x)$. By shrinking $W'_1$ one may assume that $y\in W''_1$. So as the graph of $h_a\circ h_b\circ h^{-1}$ is contained in $Z^*$ and the graph of $h_c$ is contained in $Y^*_c$ one has that $(x,y)\in Z^*\cap Y^*_e \cap W_1$. The same  is true for $x'$ so for each $e\in U_1$ we can find two different points in $ Z^*\cap Y^*_e \cap W_1$.

The same argument can be applied for finding each $i>1$, exists an open neighborhood of $c$, say $U_i$, such that for all $e\in U_i$, $|Z^*\cap Y^*_e\cap W_i|\geq 1$. Then for each $e\in V:= U_1\cap U_2\cap \ldots \cap U_n$ there is at least $n+1$ different points at $Z^*\cap Y^*_e$.

Now if 
$$Y^*_{c}\cap Z^*=\{x_1=(1,1),x_2,\ldots,x_n\},$$
then $i(x_i)$ are disctint elements of $Y_c\cap Z$ So we can list

$$Y_{i(c)}\cap Z=\{y_1=(1,1),y_2=i(x_2),\ldots,y_n=i(x_n),z_1,\ldots z_k\}.$$ Where each $z_j\in H\times H\setminus i(U)\times i(U)$. Let $z_j=(s_j,t_j)$ then if we define $U_j=s_j i(U)$ it is an infinite set contining $s_j$ and $m_j\circ i:U\to U_j$ is injective, where $m_j:G\to G$ is multiplication by $s_j$. By shrinking $U$ we may assume that the $V_j$ are pairwise disjoint. 

We can apply the same argument as before but changing the function $i$ by $m_j\circ i$  for finding an open $V_j$ set around $c$ such that for each $e$ in such a neighborhood there is at least one point at $Y_e\cap Z \cap (U_j\times U_j)$. Therefore for each $e\in V\cap V_1\cap\ldots\cap V_n$ the curve $Y_(i(e))\cap Z$ has at least $n+1$ points at $i(U)\times i(U)$ (just the image of the elements in $Z^*\cap Y_e^{*}$) and at least one at $U_i$ for each $U_i$ so in total has at least $n+k+1$ intersection points.

Therefore $E_Z$ is infinite and as $X$ is strongly minimal in $\mathcal H$ we are done.

\textit{(End of the proof of Claim \ref{claimGC})}
\end{proof}
Therefore we use Theorem \ref{groupConfig} and conclude that there is a group definable in $\mathcal H$ and there is a group configuration for $\mathcal H$ that it is inter definable with a configuration for de additive group $(K,+)$.

 \end{proof}
 So there is a group interpretable in $\mathcal M$ whose group configuration in interalgebraic (in $\mathbb K$) with the group configuration of $(K,+)$.

% Moreover we have:
 
% \begin{thm}\label{thmGoodGroup}(WARNING: I still dont have the right references for this)
 %There is a group $G=E/\sim$ interpretable in $\mathcal M$ with the following properties:
 
 %\begin{enumerate}
     
  %   \item $E$ can be taken as $X$ and $\sim$ is a $\mathcal M$ definable subset of $E\times E$ that is $\sim$ a equivalence relation on $E$ 
   %  \item Each equivalence class $[x]_\sim:=\{y\in E: x\sim y\}$ is a finite subset of $E$ for all $x\in E$
    % \item There is a $\mathcal M$ definable  set $m\subseteq E\times E\times E$ such that if we define: $$m_\sim:=\{([x_1]_\sim,[x_2]_\sim,[x_3]_\sim)\in G^3:(x_1,x_2,x_3)\in m\}$$ and for $g_1,g_2,g_3\in G$, $g_3=g_1\oplus g_2$ if and only if $(g_1,g_2,g_3)\in m_\sim$, then $\oplus$ is a group rule on $G$.
 %\end{enumerate}
 %\end{thm}
 
% Note that $E$ is strongly minimal and each class is finite, then $G$ is also strongly minimal as  interpretable group of $\mathcal M$. 

 %\begin{thm}\label{thmQEGroups}
 %There is $g\subseteq K^m$ a $\mathbb K$-definable set and a $\mathbb K$-definable bijective function $i:G'\to g'$ where $G'$ and $g'$ are cofinite subsets of $G$ and $g$ respectively 
 %\end{thm}
 %\begin{proof}
 
 %\end{proof}

 \section{Finding a Field}
 
 In this section we prove Theorem \ref{thmMultiplicative}.

First we need a Fact from \cite{HY} (Remark 7.8)
\begin{fact}\label{stronglyInternalf}
Let $H=F/\sim$ be a $\mathbb K$-interpretable group with $F\subseteq K^m$ $\mathbb K$-definable. Assume that each $\sim$-equivalence class is finite. Then, after quotient $H$ by a finite normal subgroup, there is a $\mathbb K$-definable injection $f:H'\to K^d$ for some $d$, where $H'$ is an infinite and definable subset of $H$.
\end{fact}

\begin{proof} (Proof of Theorem \ref{thmMultiplicative})

 As we said, we may assume \ref{assmDerivativeConstant}. Let  $(G,\oplus)$ an interpretable group as provided by Proposition \ref{propGroupInterpretableAdditive}. Let $e$ denote its neutral element. We may assume that $G$ is strongly minimal as a interpretable set in $\mathcal H$.
 
 Let $L\subseteq G$ and $f:L\to K^d$ as provided by Fact \ref{stronglyInternalf}, by taking translations one may assume that $e\in L$ an then group operation on $G$ can be used for defining some (partial) operation on $f(L)$ moreover we can find $a,b,c$ generic independent (in the sense of $\acl$) elements of $G$ inside of $L$ and then $f(a),f(b),f(c)$ are generic independent elements of $f(L)$.

 Thus one can use Theorem \ref{confiGroupIso} and find $U_1$, an open subset of $f(L)$ containing $f(e)$, $U_2$ a neighborhood of $0$ in $K$ and $\psi:U_2\to U_1$ a $\mathbb K$-definable bijection with analyitic inverse. Let $\phi=f^{-1}\circ \psi $ a $\mathbb K$-definable function, note that for $x,y\in U_2$  if $x+y\in U_2$ then $\phi(x+y)=\phi(z)\oplus \phi(y)$. As $U_2$ is open it contains an open ball around $0$. Replace $U_2$ for such a ball and then one can assume $U_2$ is a subgroup of $(K,+)$, so $$H:=f^{-1}(U_1)$$ is a subgroup of $G$ isomorphic to $(U_2,+)$.

% By Fact \ref{factSWAlmostQE}, given $g\in E/\sim$ generic, and $x\in g$ also generic, there is an open set $S\subseteq M^n$ containing $x$ such that $S\cap h$ is a singleton for all $h\in G$ generic. By taking translates one can assume that $S$ intersects the class of the identity on $G$. Then there is an operation defined on $S$ given by the operation on $G$. Thus we can find an group configuration inside of $S$ that is inter-algebraic to the group configuration of the additive group. Moreover by Fact \ref{factSWAlmostQE} we can assume that it is in fact inter-definable. So by Theorem \ref{confiGroupIso} there is  
%a neighborhood of $0$ say $B'\subseteq K$, an open subset $B\subseteq S$ intersecting to the identity of $G$ and an analytic map $f:B'\to B$ with analytic inverse, such that for all $x,y\in B'$ $f(x+y)_\sim=f(x)_\sim\oplus f(y)_\sim$ whenever $x+y\in B'$. By shrinking $B'$ we can assume that it is an open ball so it is a subgroup of $(K,+)$.

So $\phi:U_2\to G$ is a $\mathbb K$ definable injective group homomorphism. 

It is enough then if we prove the next general statement:

\begin{prop}\label{groupInterpretsField}
If $G$ is a $\mathbb K$-interpretable group and $\phi:U_2\to G$ is a $\mathbb K$-interpretable map such $\phi(x+y)=\phi(x)\oplus\phi(y)$ for each $x,y\in U_2$, then if $X\subseteq G\times G$ is a $\mathbb K$-definable set of dimension one that is not $G$-affine, then the structure  $\mathcal G=(G,\oplus,X)$ interprets a field
\end{prop}

\begin{proof}(Proof of Proposition \ref{groupInterpretsField})

For $Y\subseteq G\times G$ define $$Y^*:=\phi^{-1}(Y\cap H\times H)\subseteq U_2\times U_2$$ 
and for $a\in H\times H$ define $$a^*:=\phi^{-1}(a)$$

As $G$ with the induced structure is not locally modular there is some $Y\subseteq G^2$ $\mathcal M$-definable that is not $G$-affine.

In this case $Y^*$ is a curve on $U_2\times U_2$  that is not 
$(U_2,+)$-affine. 

Considerate the $\mathcal G$-definable family of curves
$$X_a:=(Y\ominus Y_a)\circ (Y\ominus Y_c)^{-1},$$ for $c\in H\times H \cap  Y$ fixed and $a\in Y$.

%Note that for $W,V\subseteq G\times G$ and $a\in W\cap H\times H$ we have:
%\begin{itemize}
 %   \item $(W\circ V)^{*}=W^*\circ V^*$
  %  \item $(W_a)^*=(W^*)_{a^*}$, and
   % \item $(V-W)^*=V^* - W^*$
    %\item $(V\oplus W)^*= V^* + W^*$
%\end{itemize}

Therefore if for $a\in Y^*$ we define,

$$X^*_a=(Y^*-Y^*_{a})\circ(Y^*-Y^*_{c^*})^{-1}.$$

Where $c^*$ is some fixed element of $Y$, we can apply Claim \ref{xIsGoodFamily} to the structure $\mathcal U$ and conclude that  $(X^*_a)_{a\in Y^*}$  is a good family of curves for $\mathcal U$. 
  
  Let $U\subseteq U_2$ and $h:U\to U_2$ as given by the data of a good family of curves. Let be
  $$\mathfrak a=(a_1,a_2,b_1,b_2,b)\in K^5$$
  as given by Lemma \ref{tupla5dim} and let 
  $$G_1=(\alpha,\beta,\gamma,p,q,r)$$
  be as given by Lemma \ref{defiDerivadaConf}. 
  
  So if $\alpha=(\alpha_1,\alpha_2)$ then $\alpha_1\in U $ so there is some $\hat \alpha_1\in H\times H$ such that $(\hat \alpha_1)^*=(\alpha_1,h(\alpha_1))$ define $\hat \alpha=(\hat \alpha_1,\hat \alpha_2)$ and the same for $\beta$, $\gamma$ and also for $p$, $q$ and $r$.
  
  We claim that 
  $$\hat G_1:=(\hat \alpha,\hat \beta,\hat \gamma, \hat p, \hat q, \hat r)$$
  is a field configuration for $\mathcal G$.
  
  Again we only have to prove the $\mathcal G$-dependence relations. So let's prove for example that $\hat q\in \acl_\mathcal G(\hat \alpha,\hat p)$

  For $\delta\in Y$ define
  \begin{equation}\label{equationDefZ2}
Z_\delta:=(X_{\hat\alpha_1}\circ X_{\hat p})\ominus (X\circ X_{\hat\alpha_1})\oplus (X\circ X_{\hat\alpha_2})\ominus (X_{\delta}\circ X)
\end{equation}

It is a $\mathcal G$ definable family of subsets of $G\times G$.
 
 %Note that $Z'_{q}((0,0))=0$ (because of Lemma \ref{additionAndCompositionLemma} and equation \ref{pIsAlg})
 
For $\delta\in Y$ define $$Z_\delta^{e}=\{x\in G:(x,e)\in Z_\delta\}$$ Since the family $(Z_\delta)_{\delta\in Y}$ is $\mathcal G$-definable with parameters $\alpha$ and $p$  it is enough to show that $$\{\delta\in Y:|Z_{\delta}^{e}|\leq|Z_{q}^{e}|\}$$ is finite. As we may assume that Morley degree of $Y$ computed in $\mathcal G$ is $1$, it is enough to show that 
$$\{\delta \in Y: |Z_\delta^0|>|Z_q^0|\}$$ is infinite.

Note that for $\delta \in Y^*$,

$$Z^*_\delta=(X^*_{\alpha_1}\circ X^*_p)-(X^*\circ X^*_{\alpha_1})+ (X^*\circ X^*_{\alpha_2})-(X^*_{\delta}\circ X^*)$$

 We proceed to prove that there is an open subset of $q$ such that for all $\delta$ in that open one has that 
 $$|Z_{\hat\delta}^e|>|Z_{\hat q}^e|$$
 
 For this list $Z^e_{\hat q}=\{(e,e)=x_1,\ldots,x_l,x_{l+1},\ldots,x_n\}$ in such a way that $x_1,\ldots,x_l\in H\times H$ and $x_{l+1},\ldots,x_n\in G\times G\setminus H\times H$.

 By Claim \ref{claimIntSube} there is an open subset $W\ni q$ such that for all $\delta \in W\setminus\{q\}$ one has that $$|(Z^*_{\delta})^e|>|(Z^*_q)^e|.$$
 Moreover 
 $$|(Z^*_\delta)^e|=|Z_{\hat\delta}^e\cap H\times H|.$$
 and the same it is true for $Z_q$.
 
 So for all $\delta \in W $ 
 $$|Z_{\hat\delta}^e\cap H\times H|>|Z_{\hat q}^e\cap H\times H|$$
 
Moreover just as in Claim \ref{claimNoPierdoCeros} we have:

  For all $i$ with $l<i\leq n$ if $V_i$ is an open set containing $x_i$, there is $W_i$ a neighborhood of $q$ such that if $\delta \in W_i\setminus\{q\}$ then $|Z_{\hat\delta}^e\cap V_i|\geq 1$

  So if we take $V_i$ disjoint not intersecting $H$ and define
  $$\bar W=W\cap W_{l+1}\cap\ldots\cap W_n,$$ 
  then for all $\delta\in \bar W\setminus q$ 
  $Z_{\hat\delta}^e$ has more than $l$ points on $H\times H$ and at least $1$ point in $V_i$ for $i>l$ so it has more than $n$ points in total. 
  
  So $\hat G_1$ is a field configuration for $\mathcal G$ therefore $\mathcal G$ interprets a field and then.
  
  \textit{(End o proof of Proposition \ref{groupInterpretsField})}
  \end{proof}
  
  Therefore as $\mathcal G$ is intepretable in $\mathcal H$, one has that $\mathcal H$ interprets a field.

\end{proof}

\chapter{General Case}

In this chapter we present a strategy for proving the general definable one dimensional case, that is we present an outline of the prove of:

\begin{conjecture}
Let $N$ be a $\mathbb K$-definable set of dimension one. Let $\mathcal N=(N,\ldots)$ a strongly minimal, non locally modular structure over $N$ that is definable in $\mathcal K$, then $\mathcal N$ interprets a field.
\end{conjecture}

For this the first step is finding a group, for doing so a procedure very similar as the one presented in Section \ref{findingAGroup} should works. In this the group structure of $G$ is just used for finding a family of curves that has no analytic points and where one can use continuity of roots. One should be able to find such a family without the group structure just as is done in the proof of Theorem 4.25 of \cite{HS}.

Moreover as the group configuration is constructed using the coeficients of the Taylor expansion of the curves, there are two cases: 

Either the first derivative is infinite (as one varies in the family) and in this case the group configuration is interalgebraic with some group configuration of $(K,+)$. Or the first derivative is finite and there is some coefficient of the taylor expansion that is infinite. 

Then by Theorem \ref{confiGroupIso} either $G$ is locally isomorphic to $(K,+)$ or it is locally isomorphic to $(M,\cdot)$. In the first case we use Proposition \ref{groupInterpretsField} and in the second we use Theorem \ref{thmMultiplicative}.

\newpage
\bibliographystyle{alpha}
\bibliography{biblio}
%\nocite{*}
\end{document}